\documentclass[12pt]{article}
\usepackage{amsmath,amsfonts,amssymb,amsthm,amscd}

\newtheorem{thm}{Theorem}[section]

\newtheorem{cor}[thm]{Corollary}
\newtheorem{lem}[thm]{Lemma}

\newtheorem{obs}[thm]{Observation}

\theoremstyle{definition}
\newtheorem{dfn}[thm]{Definition}
\newtheorem{fct}[thm]{Fact}

\newtheorem{rmk}[thm]{Remark}
\newtheorem{con}[thm]{Convention}

\newtheorem{exa}[thm]{Example}

=22cm\headheight=0cm\topskip=0cm\topmargin=0cm

\def\x{\bar {x}}

\newcommand{\M}{\sf M}
\newcommand{\N}{\sf N}
\begin{document}
\def\dis{\displaystyle}


\def\dotminus{\mathbin{\ooalign{\hss\raise1ex\hbox{.}\hss\cr
  \mathsurround=0pt$-$}}}

\begin{center}
{\LARGE{\sf  $\aleph_0$-categorical Banach spaces contain $\ell_p$
or $c_0$ }} \vspace{10mm}

{
{\bf \sf Karim Khanaki\footnote{Partially supported by IPM grants
93030032 and 93030059}}} \vspace{3mm}

{\footnotesize
  Department of science,\\
  Arak University of Technology,\\
P.O. Box 38135-1177, Arak, Iran; \\
 e-mail: khanaki@arakut.ac.ir

} \vspace{5mm}
\end{center}

{\sc Abstract.}
{\small  This paper has three parts. First, we establish some of
the basic model theoretic facts about $M_{\mathcal{T}}$, the
Tsirelson space of Figiel and Johnson \cite{FJ}. Second, using
the results of the first part, we give some facts about general
Banach spaces. Third, we study model-theoretic dividing lines in
some Banach spaces and their theories. In~particular, we show:
(1)~$M_{\mathcal{T}}$ has the \emph{non independence property}
(NIP); (2)~every Banach space that is $\aleph_0$-categorical up to
small perturbations embeds  $c_0$ or $\ell_p$ ($1\leqslant
p<\infty$) almost isometrically; consequently  the (continuous)
first-order theory of $M_{\mathcal{T}}$   does not characterize
$M_{\mathcal{T}}$, up to almost  isometric isomorphism.
 }

\medskip



{\small \tableofcontents

\bigskip
{\bf References} {\hspace{\stretch{0.01}} \bf \pageref{ref}} }

\section{Introduction}
A famous conjecture in Banach space theory had predicted that
every Banach space contains at least one of the classical spaces
 $c_0$ or $\ell_p$  for some $1\leqslant p<\infty$.
Tsirelson's example \cite{Ts} was the first space not
containing isomorphic copies any of the classical sequence
spaces, and the first space whose norm was defined implicitly
rather than explicitly.
 This phenomenon, i.e. implicit definability, later played an important role in
Banach space theory when new spaces with implicitly defined norms
yielded solutions to many of the most long standing problems in
the theory (see  \cite{Gowers}). Given the fact that all spaces
whose norms are implicitly defined do not contain any of the
classical sequences, some Banach space theorists, such as Gowers
\cite{Gowers, Gowers-weblog} and Odell \cite{Od-2}, asked the
following question:

\medskip
 \noindent\textbf{(Q1)}~``Must an `explicitly defined' Banach space contain
 $\ell_p$ or $c_0$?''
\medskip

To provide a positive answer to this question one must first
provide a precise definition of `explicitly defined space'.
Second, he/she must show that explicitly defined spaces contain
$\ell_p$ for some $1\leqslant p<\infty$, or $c_0$. Third, all
classical Banach spaces such as $L_p$ spaces, Lorentz spaces,
Orlicz spaces, Schreier spaces, etc., are explicitly defined but
the Tsirelson example and similar spaces are not (see
\cite{Gowers-weblog} for more discussions).

On the other hand, the notion of definability plays a basic and
key role in model theory and its applications in algebraic
structures. Recall that a subset $A$ of a first-order structure
$M$ is definable if $A$ is the set of solutions in $M$ of some
formula $\phi(x)$. However, the problem there is that the usual
first-order logic does not work very well for structures in
Analysis.

Chang and Keisler \cite{CK} and Henson \cite{Hen} produced the
{\em continuous logics} and the {\em logic of positive bounded
formulas}, respectively, to study structures in analysis using
model theoretic techniques. Since then, numerous attempts to
provide a suitable logic has been going on (e.g. see \cite{HenM,
HI, Ben-cat}), and eventually led to the creation of the {\em
first order continuous logic} (see \cite{BU, BBHU} for more
details). In many ways the latter logic is the best and ultimate.

One of the main purposes of this article
 is to take a step towards answering Gowers's question using
 logical tools presented in \cite{BU,
BBHU} and to prove that Banach spaces whose norms are determined
by first order statements  (i.e. their first order theories have
exactly one separable model)  contain $\ell_p$ for some
$1\leqslant p<\infty$ or $c_0$.\footnote{In this paper, when  we
say a Banach space is determined by first order statements (or
its theory) we mean its continuous theory is
$\aleph_0$-categorical in the sense of
Definitions~\ref{2definition aleph_0} or \ref{definition aleph_0}
below. That is, the first order axioms determine a unique
separable space up to isometry.} Since the dichotomy
stable/unstable (in both logics, classic and continuous) has been
widely studied in literature, and some model theoretic
properties, such as IP and SOP, in continuous logic (or Banach
space theory) have received less attention, the second goal of
this paper is to provide examples of these concepts. The third
and final purpose is to get a better understanding of
communication between both fields (model theory and Banach space
theory) so that techniques from one field might become useful in
the other.

It is worth pointing out that some of the results in the present
paper, such as  Corollary~\ref{main theorem},  seem to be folklore
to experts in the field, including Ward Henson and Alexander
Usvyatsov.  Of~course, we are not sure that this observation
itself appears somewhere in the literature. On the other hand,
some results, such as Corollary~\ref{Tsirelson->NIP} and Theorem
\ref{perturbed main theorem} and others, seem to be new even to
experts. In fact our main result is stronger than the name of the
paper.

To summarize the results of this paper, in the first part
(Section \ref{NIP on a model}), we define the notion of `NIP on a
model' and study the example of Tsirelson as a guidance. Since, by
the Krivine-Maurey theorem which asserts that every stable space
embeds some $\ell_p$ almost isometrically, Tsirelson's space is
not stable, so we consider a weaker property and ask the following
question: Does Tsirelson's space have NIP?\footnote{See
Definition~\ref{NIP-formula} for the definition of NIP on a
model.} By a result due to E. Odell \cite{Od} which asserts that
the type space of Tsirelson's space is {\em strongly separable},
we show that the answer to the latter question is
positive.\footnote{See
 Definition~\ref{definition strong separable} for the definition of strongly separable space.} On the other
hand, it is  easy to show that an $\aleph_0$-categorical space is
stable if and only if its type space is strongly
separable.\footnote{See Definitions~\ref{2definition aleph_0} and
\ref{definition aleph_0} for the definition of
$\aleph_0$-categorical structures.} Using these observations, we
conclude that the theory of Tsirelson's space (in any countable
language) is not $\aleph_0$-categorical; equivalently, this space
can not be {\em determined} by first order axioms in continuous
logic. This fact leads us to the following question:

\medskip
 \noindent\textbf{(Q2)}~Does an $\aleph_0$-categorical Banach space contain
 $\ell_p$ or $c_0$?
\medskip

In the second part of paper (Section \ref{explicit definable}),
we answer this question. In fact, the answer is ``yes", and this
is  a step towards answering the question \textbf{(Q1)} above;
furthermore we explain why $\aleph_0$-categoricity (up to small
perturbations)  is in some ways the appropriate notion for our
purpose and why this notion is not a complete answer to the
question \textbf{(Q1)}.  We prove that separable structures which
are $\aleph_0$-categorical up to small perturbations contain
$\ell_p$ or $c_0$. We will discuss this subject in detail and
offer other observations that could be interesting in themselves.

In the third part of the paper (Sections \ref{correspondence}),
we  first remark some results of \cite{K4} about
  correspondences between dividing lines in model theory and
Banach space theory which  will be used in the paper. Then, we
study model theoretic dividing lines in some Banach spaces and
their theories.

In Appendix A, we revisit $\aleph_0$-categoricity and
 in Appendix~B, we remark
an observation about Rosenthal's dichotomy communicated to us by
Michael Megrelishvili which asserts that Fact~\ref{fact3} below
is not a dichotomy in noncompact Polish case. To our knowledge
this observation itself does not appear somewhere in literature
in a clear form.

To simplify to read through the paper, we list some more
important observations and results: Theorem  \ref{perturbed main
theorem}, Observations \ref{c_0->IP}, \ref{strong->NIP}, Corollary
\ref{Tsirelson->NIP}, \ref{main theorem},
\ref{strong-not-categorical},    Examples
\ref{c_0,1}, \ref{exam2}, \ref{c_0}.
Corollary~\ref{Tsirelson->NIP}  and Theorem~\ref{perturbed main
theorem} are new.

It is worth recalling another line of research. In
\cite{Iov-2005}, Jose Iovino studied the connection between the
definability of types and the existence of $\ell_p$ and $c_0$ in
the framework of Henson's logic, i.e. the logic of positive
bounded formulas. Theorem~1.1  of \cite{Iov-2005} asserts that the
existence of enough definable types guarantees the existence of
$\ell_p$ or $c_0$ as a subspace.
After preparing of the present paper  we came to know that,
independently from us, Alexander Usvyatsov \cite{Usv} also
observed that $\aleph_0$-categorical spaces contain some classical
sequences. 

This paper is organized as follows: In the second section, we
briefly review  the notions of types in model theory and Banach
space theory. In the third section, we recall the definition of
Tsirelson's space and some facts which are used later. Also, we
present the notion NIP on a model and show that every Banach
space with strongly separable types, as well as Tsirelson's
space, has NIP.  In the forth section, we prove  that every
perturbed $\aleph_0$-categorical  space embeds $c_0$ or $\ell_p$
almost isometrically  (Theorem~\ref{perturbed main theorem}). In
the fifth section, we briefly review some model theoretic dividing
lines and their connections to some functional analysis notions,
and study some examples and investigate their model theoretic
dividing lines. In Appendix~A, we revisit $\aleph_0$-categoricity
and  in Appendix~B, we give some remarks on Rosenthal's dichotomy.

\bigskip\noindent
{\bf Acknowledgements.} I am very much indebted to  C. Ward
Henson for his kindness and his helpful and detailed comments. I
want to thank Alexander Berenstein, Gilles Godefroy, Timothy
Gowers, Richard Haydon, Tom\'{a}s Ibarluc\'{i}a, Michael
Megrelishvili and Alexander Usvyatsov for their valuable comments
and observations.
I am grateful to Thomas Schlumprecht for sending a copy of
\cite{Od} to me. I thank the anonymous referee for his detailed
suggestions and corrections; they helped to improve significantly
the exposition of this paper.

 I would like to thank the Institute for Basic
Sciences (IPM), Tehran, Iran. Research partially supported by IPM
grants  93030032 and 93030059.

\section{Preliminaries from functional analysis} \label{Preliminary}
In this section we review some basic notions from functional
analysis.  Further details can be found in \cite{AK}.

A Banach space $M=(M,\|\cdot\|)$ is a complete normed linear
space. The dual (resp. bidual) space of $M$ is denoted by $M^*$
(resp. $M^{**}$).  Recall that a Banach space $X$ is {\em
reflexive} if the mapping  $\widehat{\ }:X\to X^{**}$, given by
$\hat{x}(x^*)=x^*(x)$, is an isometry of $X$ onto $X^{**}$.

Let $1\leqslant p<\infty$. $\ell_p$ is the Banach space of all
sequences $x=(x_n)$ of reals so that
$\|x\|_p=(\sum_1^\infty|x_n|^p)^{1/p}<\infty$. $c_0$ is the Banach
space of all null sequences  $(x_n)$ under
$\|(x_n)\|_\infty=\sup_n|x_n|=\max_n|x_n|$. The (nonseparable)
Banach space $\ell_\infty$ is the space of all bounded sequences
$(x_n)$ under $\|(x_n)\|_\infty=\sup_n|x_n|$.

\medskip
Let $M$ be a Banach space. The {\em weak topology of $M$}  is the
weakest topology on $M$ such that each $x\in M^*$ is continuous.
That is, any net $(x_\alpha)$ converges weakly to $x_0\in M$ if
and only if for each $x^*\in M^*$, $x^*(x_\alpha)\to x^*(x_0)$.

\medskip
Let $X$ be a compact Hausdorff space. The space of continuous
real-valued functions on $X$ is denoted by $C(X)$. Since $X$ is a
compact space, every $f\in C(X)$ is bounded and $C(X)$ is a
complete normed linear space with the supremum norm. A net
$(f_\alpha)\in C(X)$ converges {\em pointwise} to a function
$f\in{\Bbb R}^X$ if and only if for each $x\in X$,
$|f_\alpha(x)-f(x)|\to 0$.  It is a well-known fact that for each
compact topological  space $X$, the weak topology and the
pointwise convergence topology on {\em norm-bounded} subsets of
$C(X)$ are the same (see \cite[462F]{Fremlin4}).

\begin{dfn} \label{weak compact}  Let $M$ be a Banach space and $A\subseteq M$.

\noindent (i) $A$  is  {\em relatively weakly compact} if it has
compact closure in the weak topology on $M$.

\noindent (ii)  $A$  is  {\em relatively weakly sequentially compact}
if every sequence $(x_n)$ of it has a subsequence $(y_n)$ such
that it weakly converges to an element of $M$.
\end{dfn}

Let $X$ be compact and  $A\subseteq C(X)$ norm-bounded. Then, as
the weak topology and the poinwise convergence topology on $A$
are the same,  $A$ is relatively weakly compact if it has compact
closure in the pointwise convergence topology on $C(X)$.
$A\subseteq C(X)$ is relatively weakly sequentially compact if
every sequence of it has
 a  pointwise convergent subsequence such that its limit is a
continuous function.

\begin{dfn} \label{RSC-WS-complete}
Let $M$ be  Banach space and $A$ norm-bounded subset of it.

\noindent (i) $A$ is {\em Rosenthal} if every  sequence in $A$
has a weak Cauchy subsequence. $M$ is called {\em Rosenthal} if
its unit ball is Rosenthal. In particular,  if $M=C(X)$ (where
$X$ is compact) and  $A\subseteq C(X)$ is Rosenthal, we say that
$A$  {\em is  relatively (weakly) sequentially compact in ${\Bbb
R}^X$} (short RSC). That is, every bounded sequence in $A$ has a
pointwise convergent subsequence in ${\Bbb R}^X$.

\noindent (ii)
 $A$ is said
to be {\em weakly sequentially complete} (short WS-complete)  if
every weak Cauchy sequence of $A$ has a weak limit (in $M$). $M$
is called {\em weakly sequentially complete}  if its unit ball is
weakly sequentially complete.  In particular, if
 $M=C(X)$ (where $X$ is compact) and  $A\subseteq C(X)$,
 we say that $A$ has the {\em weak sequential completeness
property} (or short SCP) if the limit of each pointwise
convergent sequence $(f_n)\subseteq A$ is continuous.
\end{dfn}

\begin{fct}[Eberlein--\v{S}mulian Theorem] \label{ES theorem} Let $M$ be  Banach space and
 $A\subseteq M$  norm-bounded. Then  the following are equivalent:
\begin{itemize}
  \item [(1)]   $A$  is   relatively weakly compact in $M$.
  \item [(2)] The following two properties hold:
   \begin{itemize}
           \item [(i)]
             $A$ is Rosenthal, and
           \item [(ii)]
              $A$ is weakly sequentially complete.
   \end{itemize}
\end{itemize}
\end{fct}

\noindent {\em Explanation.} See \cite{W} for a short proof of the
Eberlein--\v{S}mulian Theorem. (See also Theorem 1.6.3 of
\cite{AK}.)  Notice that, (2) is precisely  the condition
\textrm{B} in the main theorem of \cite{W}. Indeed, every
sequence $(x_n)$ has a subsequence $(y_n)$ such that it weakly
converges to an element of $M$ (i.e. relatively weakly
sequentially compact) if and only if (i) each sequence has a
weakly Cauchy subsequence (i.e. Rosenthal), and (ii)  every weak
Cauchy sequence has a weak limit in $M$ (i.e. weakly sequentially
complete).
 Clearly, (1)  is the condition \textrm{A}
in \cite{W}.

\medskip

\section{Types} \label{section-type}
We assume that the reader is familiar with continuous logic (see
\cite{BBHU} or \cite{BU}). We will study the notion of type in
continuous logic and its connection with the notion of type in
Banach space theory. In this paper, local types are more
important for us.

\subsection*{Types in model theory}
Although in the current paper we study unbounded metric
structures (i.e. Banach spaces) but without loss of generality we
can assume that all structures are bounded and so usual framework
of continuous logic is sufficient for our purposes.

\begin{con}
 In this paper, we study bounded metric
structures. More precisely, we study the unit balls  of   Banach
spaces. Furthermore, each formula can have an arbitrary bound,
although our focus is on the formula $\phi(x,y)=\|x+y\|$. So, we
can assume that the atomic formulas are $[0,2]$-valued.
\end{con}

As mentioned model theory notation is standard \cite{BBHU}. We
fix an $L$-formula $\phi(x,y)$ and a complete   $L$-theory $T$.
 We let $\tilde\phi(y, x) = \phi(x, y)$.    Let $M$ be an $L$-structure, $A\subseteq M$ and
$T_A=Th({M}, a)_{a\in A}$. Let $p(x)$ be a set of
$L(A)$-statements in free variable $x$. We shall say that $p(x)$
is a {\em type  over} $A$ if $p(x)\cup T_A$ is satisfiable. A
{\em complete type over} $A$ is a maximal type over $A$. The
collection of all such types over $A$ is denoted by $S^{M}(A)$,
or simply by $S(A)$ if the context makes the theory $T_A$ clear.
The {\em type of $a$ in $M$ over $A$}, denoted by
$\text{tp}^{M}(a/A)$, is the set of all $L(A)$-statements
satisfied in $M$ by $a$. If $\phi(x,y)$ is a formula, a {\em
$\phi$-type} over $A$ is a maximal consistent set of  statements
of the form $\phi(x,a)\geqslant r$ and $\phi(x,a)\leqslant s$, for
$a\in A$ and $r,s\in\mathbb{R}$. The set of $\phi$-types over $A$
is denoted by $S_\phi(A)$. Similarly, we define
$S_{\tilde\phi}(A)$.


\subsubsection*{The logic topology and $\rho$-metric}
We now give a characterization of complete types in terms of
functional analysis. Let $\mathcal{L}_A$ be the family of all
interpretations $\phi^{M}$ in $M$ where $\phi$ is an
$L(A)$-formula with a free variable $x$. Then $\mathcal{L}_A$ is
an Archimedean Riesz space of measurable functions on $M$ (see
\cite{Fremlin3}). Let $\sigma_A({M})$ be the set of Riesz
homomorphisms $I: {\mathcal L}_A\to \mathbb{R}$ such that
$I(\textbf{1}) = 1$. The set $\sigma_A({M})$ is   called the {\em
spectrum} of $T_A$. Note that $\sigma_A({M})$ is a weak* compact
subset of $\mathcal{L}_A^*$. The next proposition shows that a
complete type can be coded by a Riesz homomorphism and gives a
characterization of complete types. In fact, by compactness
theorem, the map $S^{M}(A)\to\sigma_A({M})$, defined by $p\mapsto
I_p$ where $I_p(\phi^M)=r$ if $\phi(x) = r$ is in $p$, is a
bijection.


\begin{rmk} \label{rem2} For an Archimedean Riesz space $U$  with order unit
$\textbf{1}$, write $X$ for the set of all normalized Riesz
homomorphisms from $U$ to $\Bbb R$, or equivalently,
 the positive extreme points of unit
ball of the dual space $U^*$, with its weak* topology. Then $X$ is
compact by Alaoglu's theorem and the
 natural map $u\mapsto \hat{u}:U\to \mathbb{R}^X$
defined by setting $\hat{u}(x)=x(u)$ for  $x\in X$ and $u\in U$,
is an embedding from $U$ to an order-dense and norm-dense
embedding subspace of $C(X)$.
\end{rmk}

Now, by the above remark, $\mathcal L_A$ is dense in $C(S^M(A))$.
It is easy to show that:

\begin{fct} \label{key}
Suppose that $M$, $A$ and $T_A$ are as above.
\begin{itemize}
             \item [(i)] The map $S^{M}(A)\to\sigma_A({M})$  defined by $p\mapsto I_p$ is bijective.
             \item [(ii)] $p\in S^{M}(A)$ if and only if there is an elementary
extension $N$ of $M$ and $a\in N$ such that
$p=\text{tp}^{N}(a/A)$.
\end{itemize}
\end{fct}

We equip $S^{M}(A)=\sigma_A({M})$  with the related topology
induced from $\mathcal{L}_A^*$. Therefore, $S^{M}(A)$ is a compact
and Hausdorff space. For any complete type $p$ and formula
$\phi$, we let $\phi(p)=I_p(\phi^{M})$. It is easy to verify that
the topology on $S^{M}(A)$ is the weakest topology in which all
the functions $p\mapsto \phi(p)$ are continuous. This topology
sometimes called the {\em logic topology}. The same things are
true for $S_\phi(A)$.

\begin{dfn}[The $\rho$-metric] The space $S_\phi(M)$ has another
topology. Indeed, define $\rho(p,q)=\sup_{a\in M, \|a\|\leq 1
}|\phi(p,a)-\phi(q,a)|$. Clearly, $\rho$ is a metric on
$S_\phi(M)$ and it is called the {\em $\rho$-metric } on the type
space $S_\phi(M)$.
\end{dfn}

\begin{rmk} \label{uniform top}
For unbounded logics, we define $\rho_u(p,q)=\sup_{a\in
M}|\phi(p,a)-\phi(q,a)|$.  The $\rho_u$-metric sometimes is called
the {\em uniform topology}. In this case, there is another
topology on the type space, namely the {\em strong topology}.
Indeed, define
$\rho_s(p,q)=\sum_1^k\frac{1}{2^k}\sup_{\|a\|\leqslant k}
|\phi(p,a)-\phi(q,a)|$. For the formula $\phi(x,y)=\|x+y\|$, R.
Haydon has informed us that he has constructed a Banach space $M$
for which  the space of $\phi$-types over $M$
 is strongly separable but not uniformly separable.
Despite this, for bounded continuous logic, and hence in this
paper, the $\rho$-metric and the strong and uniform topologies
are the same.
\end{rmk}

\begin{dfn}  A (compact) \emph{topometric} space is a triplet $\langle X, \frak{T},\frak{d}\rangle$,
  where $\frak T$ is a
  (compact) Hausdorff topology and $\frak d$ a metric on $X$, satisfying:
  (i) The metric topology refines the topology.
  (ii) The metric ${\frak d}:X^2\to[0,\infty]$ is lower semi-continuous,
  that is, for all $r\in{\Bbb R}^+$ the set $\{(a,b)\in X^2: {\frak d}(a,b)\leq r\}$ is closed in
  $(X^2,\tau\times\tau)$, where $\tau\times\tau$ is the product topology.
\end{dfn}

\begin{fct}  \label{topometric fact}  Let $M$ be a structure and $\phi(x,y)$ a formula. The triplet  $\langle S_\phi(M), \tau,\rho\rangle$ is
a compact topometric space where $\tau$ and $\rho$ are the logic
topology and the $\rho$-metric on the type space $S_\phi(M)$,
respectively.
\end{fct}
\begin{proof} Since $(S_\phi(M), \tau)$ is compact, it suffices to
show that $\rho$ is lower semi-continuous. (See Lemma~1.4 in
\cite{Ben-Cantor}.) Indeed, suppose that $\kappa$ is a cardinal
number and $(p_\alpha,q_\alpha)\in (S_\phi(M))^2$,
$\alpha<\kappa$, such that $\rho(p_\alpha,q_\alpha)\leq r$ for all
$\alpha<\kappa$. Let $\mathcal U$ be an ultrafilter on $\kappa$.
It is a well-known fact that, since $S_\phi(M)$ is $\tau$-compact,
there are $p,q\in S_\phi(M)$ such that the $\mathcal U$-limit of
$(p_\alpha,q_\alpha),\alpha<\kappa$ is $(p,q)$. (Recall that the
$\mathcal U$-limit of $(x_\alpha)_{\alpha<\kappa}$  is $x$ if for
every neighborhood $V$ of $x$, the set $\{\alpha:x_\alpha\in V\}$
is in $\mathcal U$. In this case, we write $\lim_{\mathcal
U}x_\alpha=x$.) Therefore,
\begin{align*}
\rho(p,q) &= \sup_{\|b\|\leq 1}|\phi(p,b)-\phi(q,b)|\\
&= \sup_{\|b\|\leq 1}\lim_{\mathcal U}|\phi(p_\alpha,b)-\phi(q_\alpha,b)|\\
&\leq \lim_{\mathcal U}\sup_{\|b\|\leq
1}|\phi(p_\alpha,b)-\phi(q_\alpha,b)|\leq r.
\end{align*}
As $\kappa$ and $\mathcal U$ are arbitrary, $\rho$ is lower
semi-continuous.
\end{proof}

Note that $(S_\phi(M),\tau)$ is the weakest topology in which all
functions $p\mapsto\phi(p,a)$, $a\in M$, are continuous, and
 all functions  $p\mapsto\phi(p,a)$, $a\in M$, are uniformly continuous relative
to  $\rho$  with a {\em modulus of uniform continuity}
$\Delta_\phi$ (see \cite{BBHU}).

\subsection*{Types in Banach space theory}
Let $M$ be a Banach space and $a$ be in the unit ball $B_M=\{x\in
M:\|x\|\leqslant 1\}$. In Banach space theory, a function
$f_a:B_M\to[0,2]$ defined by $x\mapsto\|x+a\|$ is called a {\em
trivial type over the unit ball of $M$}. A {\em type on the  unit
ball of $M$} is a pointwise limit of a family of trivial types.
The set $\mathcal{S}(M)$ of all types over the unit ball of  $M$
is called the {\em space of types}. In fact
$\mathcal{S}(M)\subseteq[0,2]^M$ is a compact topological space
with respect to the product topology. This topology is called the
{\em weak topology} on types, and denoted by $\tau'$.

$\mathcal{S}(M)$ has another topology. Define
$\rho'(f,g)=\sup_{a\in M,\|a\|\leq 1}|f(a)-g(a)|$ for $f,g\in
\mathcal{S}(M)$. Clearly, $\rho'$ is a metric on $\mathcal{S}(M)$.

\begin{rmk} Notice that, in the original defintion of types in the Banach space
theory, types are on the entire space (not just the unit ball),
and to be realized by anything in the space not just in the unit
ball. However, this does not affect the results of this paper,
because we use {\em positive} results for the original definition.
As an example, strong separability in the sense of \cite{Od}
implies strong separability in Definition~\ref{definition strong
separable} below.
\end{rmk}


\begin{fct}
The triplet $\langle \mathcal{S}(M),\tau',\rho'  \rangle$ is a
compact topometric space.
\end{fct}
\begin{proof}
The proof is a  straightforward adaptation of the proof of
Fact~\ref{topometric fact}.
\end{proof}

Let $\phi(x,y)=\|x+y\|$. Then there is a correspondence between
$S_\phi(M)$ and $\mathcal{S}(M)$. Indeed, let $a\in B_M$, and
consider the quantifier-free type $a$ over $B_M$, denoted by
$\mathrm{tp}_{\mathrm{qf}}(a/B_M)$. It is easy to verify that
$\mathrm{tp}_{\mathrm{qf}}(a/B_M)$ corresponds to the function
$f_a:B_M\to[0,2]$ defined by $x\mapsto\|x+a\|$. Also, there is a
correspondence between $\mathrm{tp}_{\mathrm{qf}}(a/B_M)$ and the
Riesz homomorphism $I_a:{\mathcal{L}_M}\to[0,2]$ defined by
$I_a(\phi(x,y))=\phi^{\M}(x,a)$ for all $x\in B_M$. To summarize,
the following maps are bijective:

$$\mathrm{tp}_{\mathrm{qf}}(a/B_M)\rightsquigarrow f_a$$
$$\mathrm{tp}_{\mathrm{qf}}(a/B_M)\rightsquigarrow I_a$$

Now, it is easy to check that their closures are the same:

\begin{fct}
Assume that $M$ is a Banach space and $\phi(x,y)=\|x+y\|$. The
compact topometric spaces $\langle S_\phi(M),\tau,\rho \rangle$
and $\langle \mathcal{S}(M),\tau',\rho' \rangle$ are the same.
More exactly, $(S_\phi(M),\tau)$ and $(\mathcal{S}(M),\tau')$ are
homeomorphic, and $(S_\phi(M),\rho)$ and $(\mathcal{S}(M),\rho')$
are isometric.
\end{fct}

\section{Tsirelson's space and NIP} \label{NIP on a model}
 In this section we show that the  Tsirelson space has a property weaker than stability, namely the non independence property (NIP).
First we remind the notion NIP on a model from \cite{K3}.

\subsection*{NIP on a model}

\begin{dfn} \label{NIP-formula}
(i) Let $\M$ be a structure, and $\phi(x,y)$ a formula. We say
that $\phi(x,y)$ has the {\em independence property (short IP)  on
$M$} if there are a sequence $(a_n)\subseteq M$, and $r>s$, such
that for each {\em finite} disjoint subsets $E,F$ of ${\Bbb N}$:
$$\Big\{b\in M:\big( \bigwedge_{n\in E}\phi^{\M}(a_n,b)\leqslant
  s\big)\wedge\big(\bigwedge_{n\in F}\phi^{\M}(a_n,b)\geqslant r\big)\Big\}\neq\emptyset.$$
  In this case,
  we say that $(\phi^{\M}(a_n,y):i<\omega)$ is an independent sequence.
 We say that $\phi(x,y)$ has  {\em  NIP}
   on $M$ if it does not have IP on $M$.
  \noindent
  (ii) A complete theory $T$ has IP if
there are a formula $\phi$ and a model $M$ of it such that $\phi$
has IP on $M$, and otherwise it is said that $T$ has NIP.
\end{dfn}

\begin{lem} \label{equivalence}   Let $M$ be a structure, and $\phi(x,y)$ a formula.
Then the following are equivalent.
\begin{itemize}
            \item [(i)]  $\phi(x,y)$ has NIP on $M$.
           \item [(ii)] For each sequence $(a_n)\subseteq M$, each elementary extension ${\N}\succeq{\M}$, and
                       $r>s$, there are some {\em possibly infinite} disjoint subsets $E,F$ of ${\Bbb N}$ such that $$\Big\{b\in N: \big( \bigwedge_{n\in E}\phi^{\N}(a_n,b)\leqslant
                        s\big)\wedge\big(\bigwedge_{n\in F}\phi^{\N}(a_n,b)\geqslant r\big)\Big\}=\emptyset.$$
           \item [(iii)] For each sequence $\phi(a_n,y)$ in the set
                        $A=\{\phi(a,y):S_{\tilde\phi}(M)\to{\Bbb R}~|~a\in M\}$, where $S_{\tilde\phi}(M)$ is
                        the space of all complete ${\tilde\phi}$-types on $M$, and $r>s$
                        there are some {\em finite} disjoint subsets $E,F$ of ${\Bbb N}$ such that  $$\Big\{q\in S_{\tilde\phi}(M):\big( \bigwedge_{n\in E}\phi(a_n,q)\leqslant
                        s\big)\wedge\big(\bigwedge_{n\in F}\phi(a_n,q)\geqslant r\big)\Big\}=\emptyset.$$
           \item [(iv)] The condition (iii) holds for
                          {\em arbitrary} disjoint subsets $E,F$ of ${\Bbb N}$.
           \item [(v)] Every sequence $\phi(a_n,y)$ in $A$ has a pointwise convergent subsequence.
\end{itemize}
\end{lem}
\begin{proof}
See Lemma~3.12 in \cite{K3} for a proof.
\end{proof}


\begin{rmk} (i) A formula $\phi$ has NIP for a theory $T$ (in the obvious sense) iff it has NIP on every
model $\M$ of $T$ iff it has NIP on some approximately
$\omega$-saturated model of $T$. (See Definition~1.3 in \cite{BU}
for the definition of approximate $\omega$-saturation.)
\newline
(ii) Note that for each separable structure $\M$, every sequence
$\phi^{\M}(a_n,y):M\to\Bbb R$ has a convergent subsequence (see
Lemma~\ref{diagonal-lemma} below), but this does not imply that
every sequence $\phi(a_n,y):S_{\tilde\phi}(M)\to\Bbb R$ in $A$ has
a convergent subsequence (see Examples~\ref{c_0,1} and \ref{c_0}
below).
\newline
(iii) The notion of NIP on a model was introduced in \cite{K3}.
Recently, in \cite{KP} some equivalences of this notion are
given. For example, it is shown that a formula $\phi$ has NIP on
$\M$ iff every coheir of a $\phi$-type over $\M$ is Borel (Baire
1) definable.
\end{rmk}

\begin{dfn} A Banach space (or Banach structure) $\M$ has  {\em NIP} (resp. \emph{IP}) if
the formula $\phi(x,y)=\|x+y\|$ has NIP (resp. IP) on $\M$.
\end{dfn}

 The following example shows that $c_0$ has  IP. Also, we will see
later that this example shows that Fact~\ref{fact3} below is not
dichotomy in noncompact spaces (see Example~\ref{c_0} below). We
thank Michael Megrelishvili for communicating to us the example.

\begin{exa} \label{c_0,1}
 Let $c_0=\{x=(x_n)_n\in\ell_\infty:\lim_nx_n=0\}$ with
$\|x\|=\sup_{n\in{\Bbb N}}|x_n|$. Let $B_{c_0}$ be the unit ball
of $c_0$, i.e. $B_{c_0}=\{x\in c_0:\|x\|\leqslant 1\}$. For each
$a\in B_{c_0}$, define  $f_a:B_{c_0}\to[0,2]$ by $x\mapsto
\|x+a\|$. Let ${\mathcal F}_{c_0}=\{f_a:a\in B_{c_0}\}$.  Then the
family ${\mathcal F}_{c_0}$ contains an independent sequence.
Indeed, for each $n$, define $f_n(x)=\|x+e_n\|$ for all $x$ where
$(e_n)_n$ is the standard basis of $c_0$. We show that $(f_n)_n$
is an independent sequence: for every   finite disjoint subsets
$I$ and $J$ of $\Bbb N$ (the naturals) define a binary vector
$x\in c_0$ as the characteristic function of $J$: $x_j=1$ for
every $j\in J$ and $x_k=0$ for every $k\notin J$ (in particular,
for every $i\in I$). Then $x\in c_0$, $\|x\| = 1$ and $f_i(x)=1$
for every $i\in I$, $f_j(x)=2$ for every $j\in J$. Therefore,
 $(f_n)_n$ is an independent sequence.
\end{exa}

\begin{obs} \label{c_0->IP}  Suppose that $X$ is a Banach space structure
containing a subspace isomorphic to $c_0$.Then $X$ has the
independence property (IP).
\end{obs}

\begin{proof}
  Recall that every space isomorphic to $c_0$
 has subspaces that embed  $c_0$ almost isometrically. (See Proposition~2.e.3 in
 \cite{LT}.) So, the sequence $(f_n)_n$ in Example~\ref{c_0,1} works
 well.
\end{proof}

The converse does not hold, and we will give a counterexample  in
a future work.


\bigskip
We now recall the notion of non order property (NOP) or stability
on a model.

\begin{dfn} \label{stable-formula}
Let $\M$ be a structure, and $\phi(x,y)$ a formula. The following
are equivalent and in any of the cases we say that $\phi(x,y)$
{\em is stable on $M$} (or {\em has NOP on $M$}).
\begin{itemize}
             \item [(i)] Whenever $a_n,b_m\in M$ form two sequences we have $$\lim_n\lim_m\phi(a_n,b_m)=\lim_m\lim_n\phi(a_n,b_m),$$
                               when the limits on both sides exist.
             \item [(ii)]  The set $A=\{\phi(x,b):S_x(M)\to{\Bbb R}~|b\in M\}$ is relatively weakly compact in
             $C(S_x(M))$. (Cf. Definition~\ref{weak compact}
             above.)
\end{itemize}
\end{dfn}

\begin{rmk}
(i) The equivalence (i)~$\Leftrightarrow$~(ii) of
Definition~\ref{stable-formula} is called the Eberlein--
Grothendieck  criterion (see~\cite{K3}). In \cite{KM}, stability
is only defined for the  formula $\|x+y\|$, and it is called
stability for the whole  space (and not NOP for the formula
$\|x+y\|$). Definition~\ref{stable-formula}(i) was based on Shelah
general notion of NOP or stability of a formula. The general
analogue of NOP in continuous setting appeared for the first time
in \cite{BU}.

\noindent (ii) A formula $\phi$ has NOP for a theory $T$ (in the
obvious sense) iff it has NOP on every model $\M$ of $T$. A
theory has NOP (or is stable) if every formula has NOP for the
theory.

\end{rmk}

\begin{dfn} A Banach space (or Banach structure) $\M$ is {\em stable} (or {\em has NOP}) if
the formula $\phi(x,y)=\|x+y\|$ is stable on $\M$.
\end{dfn}

 The
following is a consequence of the well-known compactness theorem
of Eberlein and \v{S}mulian (Fact~\ref{ES theorem}). Indeed, if
the set $A$ above is relatively weakly compact in $C(S_x(M))$,
then every sequence in $A$ has a pointwise convergent subsequence
in ${\Bbb R}^{S_x(M)}$, namely $A$ is  relatively sequentially
compact in ${\Bbb R}^{S_x(M)}$.  (Cf.
Definition~\ref{RSC-WS-complete}(i).)

\begin{fct}[\cite{K3}, Fact~3.4]
Let $\M$ be a structure, and $\phi(x,y)$ a formula. If $\phi$ is
stable on $\M$ then $\phi$ has NIP on $\M$. In particular, every
stable Banach space has NIP.
\end{fct}

As $L_p[0,1]$  ($1\leqslant p<\infty$) embeds $\ell_p$, and the
former is stable (cf. \cite[Theorem~17.11]{BBHU}), so $\ell_p$ is
stable. Therefore, by the above fact, $\ell_p$  ($1\leqslant
p<\infty$) has NIP.

\subsection*{Tsirelson's space} \label{Tsirelson's space}
 Although we only use some known properties of
Tsirelson's example \cite{Ts} or actually the dual space of the
original example as described by Figiel and Johnson \cite{FJ},
for the sake of completeness, we present Tsirelson's space and
remind some its properties which are used in this paper. The key
property for our purpose will be presented in
Fact~\ref{Odell-fact}.

A collection $(E_i)_1^n$ of finite subsets of natural numbers is
called {\em admissible} if $n\leqslant E_1<E_2<\cdots<E_n$. By
$E<F$ we mean $\max E<\min F$ and $n\leqslant F$ means
$n\leqslant\min F$. For $x\in c_{00}$ and $E\subseteq {\Bbb N}$
by $Ex$ we mean the restriction of $x$ to $E$, i.e. $Ex(n)=x(n)$
if $n\in E$ and $0$ otherwise.

Now we recall the definition $M_{\mathcal{T}}$ (the Tsirelson of
\cite{FJ}). For all $x\in c_{00}$ set
$$\|x\|=\max\Big(\|x\|_\infty,~\sup\Big\{\frac{1}{2}\sum_{i=1}^n\|E_ix\|: (E_i)_1^n \text{ is admissible} \Big\}\Big).$$
$M_{\mathcal{T}}$ is then defined to be the completion of
$(c_{00},\|\cdot\|)$. Note that this norm is given implicitly
rather than explicitly.

One can define this norm on $c_{00}$ by induction. Indeed, for
$x\in c_{00}$, set $\|x\|_0=\|x\|_\infty$ and inductively
$$\|x\|_{n+1}=\max\Big(\|x\|_n,~\sup\Big\{\frac{1}{2}\sum_{i=1}^n\|E_ix\|_n: (E_i)_1^n \text{ is admissible} \Big\}\Big).$$
Then $\|x\|=\lim_n\|x\|_n$ is the desired norm.

\begin{fct}[\cite{AK}, Theorem~11.3.2] \label{properties}
Tsirelson's space has the following properties.
\begin{itemize}
             \item [(i)] $M_{\mathcal{T}}$ is reflexive.
             \item [(ii)] $M_{\mathcal{T}}$ does not contain a subspace
             isomorphic to $c_0$ or $\ell_p$ ($1\leqslant
             p<\infty$).
\end{itemize}
\end{fct}

In addition, it is easy to verify that $(e_n)$ is a normalized
{\em unconditional basis} for $M_{\mathcal{T}}$ and any {\em
spreading model} of $M_{\mathcal{T}}$ is isomorphic to $\ell_1$.
See \cite{AK} for definitions and proofs.

The following deep result, due to Krivine and Maurey \cite{KM},
gives a partially answer to the main question of the present
paper. First, recall that:

\begin{dfn} \label{almost isometric}  Let $X, Y$ be two Banach
spaces. (i) $X$ and $Y$  are {\em $\lambda$-isomorphic} (for some
$\lambda\geqslant 1$), if there is a linear isomorphism $T$  from
$X$ onto $Y$ such that for all $x\in X$,
$\lambda^{-1}\|x\|\le\|T(x)\|\le \lambda\|x\|$.

\noindent  (ii)  $X$ and $Y$ are called {\em almost isometric} if
for each $\lambda>1$, they are $\lambda$-isomorphic.

\noindent (iii) {\em  $X$ embeds $Y$ almost isometrically}, if for
each $\lambda>1$ there is a subspace $X_\lambda$ of $X$ such that
$X_\lambda$ and $Y$ are  $\lambda$-isomorphic.
\end{dfn}

\begin{fct}[Krivine--Maurey theorem] \label{Krivine--Maurey}  Every (separable) stable Banach space
embeds $\ell_p$ almost isometrically, for some $1\leqslant
p<\infty$.
\end{fct}

\begin{cor}  \label{unstable}
$M_{\mathcal{T}}$ is not stable.
\end{cor}

\begin{proof}
Immediate by Facts~\ref{properties} and \ref{Krivine--Maurey}.
\end{proof}

A Banach space is  {\em weakly stable} if the condition~(i) in
Definition~\ref{stable-formula} holds whenever $(a_n)$ and
$(b_m)$ are both weakly convergent. It is proved that every weakly
stable space embeds $c_0$ or $\ell_p$, for some $1\leqslant
p<\infty$, almost isometrically (see \cite{ANZ}).  To summarize,
$M_{\mathcal{T}}$ is not stable, even worse, it is not weakly
stable.

\subsection*{Strong separability}  \label{strong
separability} We will show that Tsirelson's space has NIP.
Indeed, we show that every (separable) Banach structure with
strongly separable space of types has NIP. Then the desired
result will be achieved from a deep result of Odell
(Fact~\ref{Odell-fact} below).


\begin{dfn} \label{definition strong separable} We say a (separable) Banach space is {\em strongly separable},
or equivalently {\em its type space is strongly separable}, if
$(S_\phi(M),\rho)$ is separable where $\phi(x,y)=\|x+y\|$ and
$\rho$ is the $\rho$-metric on $\phi$-types.
\end{dfn}

 Recall that, by the Krivine--Maurey theorem, the Tsirelson's space is not stable (see Corollary~\ref{unstable} above).
By Odell's result below, it easy to show that this space has a
weaker property, namely NIP. For this, we recall some notion and
fact.

\medskip
 Recall that a function $f$
form a metric space $X$ to a metric space $Y$ is $k$-Lipschitz
($k\geqslant 0$) if for all $x,y$ in $X$, $d(f(x),f(y))\leqslant
k\cdot d(x,y)$. Now we give the following easy lemma.

\begin{lem}  \label{diagonal-lemma}
Assume that $(X,d)$ is a metric spaces, and $F$ a bounded family
of $1$-Lipschitz functions from $X$ to $\Bbb R$. If $X$ is
separable then every sequence in $F$ has a (pointwise) convergent
subsequence.
\end{lem}

\begin{proof} The proof is an easy diagonal argument.
\end{proof}



\begin{obs} \label{strong->NIP}
Assume that $M$ is a separable Banach structure and
$\phi(x,y)=\|x+y\|$.   If the type space $S_\phi(M)$ is strongly
separable, then $\phi$ has NIP on $M$.
\end{obs}

\begin{proof}
 Recall that $(S_\phi(M),\tau,\rho)$ is a
 compact topometric space where $(S_\phi(M),\tau)$ is compact and
 Polish, and  $(S_\phi(M),\rho)$ is a complete metric space. (Here, $\tau$ is the logic topology and $\rho$
 is the $\rho$-metric.)  Note that $\rho$ is separable since the type space is strongly separable. Also, the functions of the form
 $\phi(\cdot,a),a\in M$, are $1$-Lipschitz with respect to $\rho$.
 Now, use  Lemma~\ref{diagonal-lemma} and the equivalence
 (i)~$\Leftrightarrow$~(v) of
 Lemma~\ref{equivalence}.
\end{proof}

In \cite[Proposition~1]{Od}, it is shown that the type space of a
 separable stable Banach
spaces is separable with respect to the metric $\rho_s$
(cf.~Remark~\ref{uniform top}).
 For the unit ball of a separable and
stable space $M$, this is actually a consequence of the
separability of $M$ and the definability of types in stable
models (see \cite{Ben-Gro} and
\cite{K4}). 
 Edward Odell showed that the converse does not hold:

\begin{fct}[Odell, \cite{Od}] \label{Odell-fact}
Tsirelson's space is strongly separable. Furthermore, its type
space in unbounded continuous logic is uniformly separable (see
Remark~\ref{uniform top} above).
\end{fct}

 Let $M$ be a Banach space, $\phi(x,y)=\|x+y\|$, $X=S_{\tilde\phi}(M)$
the space of complete $\tilde\phi$-types on $M$. For $a\in M$,
define $\phi(a,y):X\to[0,2]$ by $q\mapsto \phi(a,q)$. (Notice that
$\phi(a,y):X\to[0,2]$, $a\in M$, is continuous.) We say that {\em
the space of definable predicates on $M$ is weakly sequentially
complete} if the set $F=\{\phi(a,y):X\to[0,2]~|a\in M\}$  has the
weak sequential completeness property  (cf.
Definition~\ref{RSC-WS-complete}(ii)).

The following observation is new, though it is really just a
corollary of Odell's deep result above.

\begin{cor} \label{Tsirelson->NIP} (i)  Tsirelson's  space has NIP.

\noindent
(ii) The space of definable predicates on Tsirelson's
space is not weakly sequentially complete.
\end{cor}

\begin{proof}
(i) Immediate, since the type space of Tsirelson's space is
strongly separable (see Fact~\ref{Odell-fact} above).
\newline
(ii) Note that, by the Eberlein--\v{S}mulian theorem
(Fact~\ref{ES theorem}), NIP and the weak sequential completeness
property is equivalent to stability. (See also Theorem~4.3 of
\cite{K3}.)
\end{proof}

It seems to be open whether the theory of Tsirelson's space has
NIP.

\begin{rmk} \label{Gowers example}
Note that one can not expect a converse to
Observation~\ref{strong->NIP}. There are some indications that the
{\em tree} Banach spaces can be counterexamples. With the previous
observations, we divide separable Banach spaces to some classes
such as stable, containing some $\ell_p$ ($1\leqslant p<\infty$),
with strongly separable type space, with NIP, not containing
$c_0$, and these classes form a hierarchy
$$\textrm{stable}\varsubsetneq\textrm{strongly separable}\subset_{?}\textrm{NIP}\subset_{?}\textrm{ not containing $c_0$}$$
Two questions arise. Must a space not containing $c_0$ have NIP?
Must a space with NIP have strongly separable types?  These
questions seem to be open, although we know that the answer to
one of these questions is certainly negative. Indeed, in the 90's
Gowers \cite{Gowers-94} has constructed a space $X_G$ not
containing $\ell_p$ and no reflexive subspace. By the result of
Haydon and Maurey \cite{HM}  which asserts that the spaces with
strongly separable types contain $\ell_1$ or have a reflexive
subspace, the type space of $X_G$ is not strongly separable and
this space does not contain $c_0$, so \{strongly separable\}
$\varsubsetneq$ \{not containing $c_0$\}. Also, we strongly
believe that the answers to both questions are negative.  We will
give sharper answers in a future work.
\end{rmk}

\section{$c_0$- and $\ell_p$-subspaces of perturbed $\aleph_0$-categorical
spaces} \label{explicit definable}  As we saw earlier the norm of
Tsirelson's space is given implicitly rather than explicitly. An
easy  observations  (see Corollary~\ref{strong-not-categorical}
below) shows that the norm of Tsirelson's space can not be
characterized by axioms in continuous logic. This and the fact
that Tsirelson's space does not contain $c_0$ or $\ell_p$ lead us
to a natural question as Gowers and Odell asked: Must an
``explicitly defined" Banach space contain $c_0$ or $\ell_p$?
(See \cite{Gowers, Gowers-weblog, Od-2}.) Now, we give an
explicit definition of a similar alternative of an ``explicitly
defined" Banach space, and give a positive answer to the above
question.

\begin{con} In this section the language is the usual language of Banach spaces (in continuous logic). Otherwise, we explicitly state  what is our desired language.
\end{con}


\subsection*{$c_0$- and $\ell_p$-types}
\begin{dfn} If $p\in[0,\infty)$ and $\epsilon>0$, the following set of statements
with countable variables $\x=(x_0,x_1,\ldots)$ will be called the
{\em $\epsilon-\ell_p$-type}: for every natural number $n$, and
scalars $r_0,\ldots,r_n$,
$$(1+\epsilon)^{-1}\Big\|\sum_0^n r_ix_i\Big\|\leqslant\Big\|\big(\sum_0^n|r_i|^p\big)^{\frac{1}{p}}x_0\Big\|\leqslant (1+\epsilon)\Big\|\sum_0^n r_ix_i\Big\|.$$

The following set of statements with countable variables
$\x=(x_0,x_1,\ldots)$ will be called the {\em
$\epsilon-c_0$-type}: for every natural number $n$, real number
$\epsilon>0$ and scalars $r_0,\ldots,r_n$,
$$(1+\epsilon)^{-1}\Big\|\sum_0^n r_ix_i\Big\|\leqslant\Big\|\big(\max_{0\leqslant i\leqslant n}|r_i|\big)x_0\Big\|\leqslant (1+\epsilon)\Big\|\sum_0^n r_ix_i\Big\|.$$
\end{dfn}

\begin{rmk} The notion of $\epsilon-\ell_p$-type ($\epsilon-c_0$-type)
is a formal definition, and is different from the notion of types
in model theory or Banach space theory. Recall that if $T$ is a
complete theory and there exists a model $M$ of $T$ such that the
$\epsilon-\ell_p$-type ($\epsilon-c_0$-type) is realized in $M$,
then it is a type in the sense of model theory, i.e. it belongs to
$S_{\x}(\emptyset)$. However, the following deep classical
theorem, due to Krivine, guarantees that for every Banach space
there exists some $p$ such that the notion of
$\epsilon-\ell_p$-type is really a type in the sense of model
theory.
\end{rmk}

\subsection*{Krivine's theorem}
We recall  the well-known theorem of Krivine.

\begin{dfn}
Let $p\in[1,\infty]$ and $(x_i)$ a sequence of elements of some
Banach space. We say that $\ell_p$ (resp. $c_0$) is block finitely
represented in $(x_i)$ if $p\in[1,\infty)$ (resp. $p=\infty$) and
for every $\epsilon>0$ and positive integer $n$, there are $n+1$
finite subsets $F_0,\ldots,F_n$ of the positive integers $N$ with
$\max F_j<\min F_{j+1}$, for all $0\leqslant j\leqslant n-1$ and
elements $y_0,\ldots,y_n$ with $y_j$ in the linear span of $\{x_i
: i\in F_j\}$ for all i so that for all scalars $r_1,\ldots,r_n$,
$$(1-\epsilon)\Big\|\big(\sum_0^n|r_i|^p\big)^{\frac{1}{p}}y_0\Big\|\leqslant\Big\|\sum_0^n r_iy_i\Big\|\leqslant(1+\epsilon)\Big\|\big(\sum_0^n|r_i|^p\big)^{\frac{1}{p}}y_0\Big\|$$
(where $(\sum_0^n|r_i|^p)^\frac{1}{p}=\sup|r_i|$, if $p=\infty$).
\end{dfn}

In the year in which Tsirelson's example appeared in print
\cite{Ts}, J.-L. Krivine \cite{Kr} publised his celebrated
theorem.

\begin{fct}[Krivine's theorem] Let $(x_n)$ be a sequence in a
Banach pace with infinite-dimensional linear span. Then either
there exists a $1\leqslant p<\infty$ so that $\ell_p$ is block
finitely represented in $(x_n)$ or $c_0$ is block finitely
represented in $(x_n)$.
\end{fct}

\subsection*{Main theorem}\label{perturb}
In this section, we prove the main theorem of the paper;  that is,
any (suitable) perturbed $\aleph_0$-categorical space contains
some classical sequences $c_0$ or $\ell_p$.

Let $M,N$ be two Banach spaces and $a\in M$ and $b\in N$. Recall
from Definition~\ref{almost isometric} that, for $\epsilon>0$, a
$(1+\epsilon)$-isomorphic between $(M,a)$ and $(N,b)$ is a linear
map $f:M\to N$ such that $f(a)=b$ and $\|f\|,\|f^{-1}\|\leqslant
1+\epsilon$. Let us fix a complete theory $T$ (in a countable
language of Banach spaces) and a monster model $\mathcal{U}$ of
it. In the following $S_n(T)$ denotes the set of $n$-types over
$\emptyset$.

\begin{dfn}[Banach--Mazur perurbation]  \label{Banach--Mazur perurbation}
The Banach--Mazur perturbation system $\mathfrak{p}_{_{BM}}$ (for
theory $T$) is the perturbations system defined by the
Banach--Mazur distance $d_{BM}$ in the sense of Definition~1.23
in \cite{Ben-perturb}, that is, for each $\epsilon>0$,
$\mathfrak{p}_{_{BM}}(\epsilon)=\{(p,q)\in S_n(T)^2:n<\omega,~
d_{BM}(p,q)<\epsilon\}$ where
$$d_{BM}(p,q)=\inf\left\{\log(1+\epsilon):
     \begin{aligned}[c]
      & \text{there is a $(1+\epsilon)$-isomorphism between $(\mathcal{U},a)$ } \\
      & \text{ and $(\mathcal{U},b)$, and $\mathcal{U}\vDash p(a),q(b)$}
    \end{aligned}\right\}$$
\end{dfn}

See also Iovino's paper  \cite{Iovino-stability-I}.

\begin{dfn}[$\mathfrak{p}_{_{BM}}$-isomorphism] \label{perturb isomorphism}
\begin{itemize}
             \item [(i)] Two separable models $M$, $N$ are called $\mathfrak{p}_{_{BM}}$-isomorphic if for each $\epsilon>0$ there
             is a bijective map $f:M\to N$ such that  $(\text{tp}^M(a),\text{tp}^N(f(a)))\in\mathfrak{p}_{_{BM}}(\epsilon)$ for all $a\in M^n$, $n<\omega$.
             \item [(ii)] A theory $T$ is called $\mathfrak{p}_{_{BM}}$-$\aleph_0$-categorical if every two separable models $M, N\vDash T$
             are $\mathfrak{p}_{_{BM}}$-isomorphic. A separable Banach space
             is called $\mathfrak{p}_{_{BM}}$-$\aleph_0$-categorical if
             its continuous first order theory (in the usual
             language) is  $\mathfrak{p}_{_{BM}}$-$\aleph_0$-categorical.
\end{itemize}
\end{dfn}

Clearly, every two models of a
$\mathfrak{p}_{_{BM}}$-$\aleph_0$-categorical theory are  almost
isometric  (cf. Definition~\ref{almost isometric} above).

We want to give a general result to any perturbation system, so
we generalizes the  above notions.  Let $\frak p$ be an arbitrary
perturbation system in the sense of
\cite[Definition~1.23]{Ben-perturb}. The notions $\frak
p$-isomorphism and $\frak p$-$\aleph_0$-categoricity are  defined
similar to Definition~\ref{perturb isomorphism}, by replacing
$\mathfrak{p}_{_{BM}}$ with $\mathfrak{p}$. Similarly, a separable
Banach space is called $\mathfrak{p}$-$\aleph_0$-categorical if
its continuous first order theory (in the usual language) is
$\mathfrak{p}$-$\aleph_0$-categorical.

\medskip
The following terminology is not standard, so we single it out.
Before that, we recall some notations. If $p(x)$ is any partial
type and $\epsilon>0$ then $p(x^\epsilon)$ denotes the partial
type $\exists x'(p(x')\wedge d(x,x')\leq\epsilon)$. If $\bar x$
is a countable tuple of variables then $p(\bar x^\epsilon)$ means
$p( x_0^\epsilon, x_1^\epsilon , \ldots)$, where the metric is
supermum metric on countable tuples.  For a perturbation system
$\frak p$, the ${\frak p}(\epsilon)$-neighbourhood around $p(\bar
x)$ is defined as follows: $p^{{\frak p}(\epsilon)}(\bar
x)=\{q(\bar x):(p(\bar x),q(\bar x))\in{\frak p}(\epsilon)\}$.
(Cf. Definition~1.1 of \cite{Ben-perturb}.)

\begin{dfn}[$\frak p$-saturation] \label{p-saturation} Let $\frak p$ be a perturbation
system. We will say that a structure $M$ is {\em
$\mathfrak{p}$-saturated} if for every type $p(\bar x)\in
S_\omega(\emptyset)$ (in  countable variable $\bar x$), and
$\epsilon>0$, the partial type $p^{{\frak p}(\epsilon)}(\bar
x^\epsilon)$ is realised in $M$.
\end{dfn}

Notice that the notion $\mathfrak{p}$-saturation above is weaker
than {\em $\mathfrak{p}$-approximately $\aleph_0$-saturation}  in
the sense of \cite[Definition~2.2]{Ben-perturb}. Indeed, recall
from Remark~2.4 of \cite{Ben-perturb} that the latter notion can
be defined for $\omega$-types in $S_\omega(\emptyset)$ (i.e, types
with $\omega$-tuples).

\begin{dfn}[Perturbed embedding] Let $\frak p$ be a perturbation
system,  $M$ a Banach space and $p(\bar x)$ the $0-\ell_p$-type
(or  $0-c_0$-type). We say that {\em  $M$
 embeds $\ell_p$ (or $c_0$)  $\frak p$-approximately} if for every
 $\epsilon>0$, the
partial type $p^{{\frak p}(\epsilon)}(\bar x^\epsilon)$ is
realised in $M$.
\end{dfn}

\medskip
Now we are ready to give the main result of the paper.

\begin{thm}[Main theorem] \label{perturbed main theorem} Let $\mathfrak{p}$ be perturbation system and $M$  $\mathfrak{p}$-$\aleph_0$-categorical
space. Then, $M$ embeds $c_0$ or $\ell_p$ ($1\leqslant p<\infty$)
$\frak p$-approximately.  Moreover, $M$ embeds  $\ell_2$ $\frak
p$-approximately. In particular, if $\frak p$ is the
Banach--Mazur perturbation system $\mathfrak{p}_{_{BM}}$, then
$M$ embeds  $\ell_2$ almost isometrically.

\end{thm}
\begin{proof} Let $M$ be a separable model of $T$. By Krivine's
theorem, some $\ell_p$ (or $c_0$) is block finitely representable
in $M$. So, for every $\epsilon>0$, the $\epsilon-\ell_p$-type is
a type in $S_\omega(T)$(=$S_\omega(\emptyset)$). (Note that in
this case the $\epsilon-\ell_p$-type is a partial type, but we
can expand it to a complete type in  $S_\omega(T)$.)
 As $\epsilon$ is
arbitrary, the $0-\ell_p$-type is a complete type. Let $p(\bar x)$
be this complete type.
By 
 Lemma~3.4 of  \cite{Ben-perturb},
 $M$ is  $\frak p$-saturated. (Indeed, recall that the notion $\mathfrak{p}$-approximately
 $\aleph_0$-saturation in  \cite[Definition~2.2]{Ben-perturb} is stronger than $\frak p$-saturation.)
By the $\frak p$-saturation, for every $\epsilon>0$, $p^{{\frak
p}(\epsilon)}(\bar x^\epsilon)$ is realised in $M$, and so $M$
 embeds $\ell_p$  $\frak p$-approximately.
Moreover, it is known that $\ell_2$ is finitely representable in
$c_0$ and $\ell_p$. (Indeed, in the case of $c_0$, every Banach
space is finitely representable in $c_0$ (see
\cite[Example~12.1.2]{AK}). If $\ell_p$ is finitely
representable, then so is $L_p$ (\cite[Proposition~12.1.8]{AK}),
and since $\ell_2$ is isometric to a subspace of $L_p$
(\cite[Theorem 6.4.12]{AK}), $\ell_2$ is finitely
representable.)  So $\epsilon-\ell_2$-type
  is a partial type in the sense of model theory, and so $M$ embeds $\ell_2$ $\frak p$-approximately.

 In particular, if $\frak p$ is the Banach--Mazur perturbation
system $\mathfrak{p}_{_{BM}}$, $M$ embeds $\ell_p$ (or $c_0$)
${\frak p}_{_{BM}}$-approximately.  That is, for each $\epsilon>0$
there is a sequence $\bar a\in M$ such that it realise the
partial type $p^{{\frak p}_{_{BM}}(\epsilon)}(\bar x^\epsilon)$.
This means that there is a sequence $\bar b$ in an elementary
extension of $M$ such that $d(\bar a,\bar b)\leq\epsilon$ and
$\models p^{{\frak p}_{_{BM}}(\epsilon)}(\bar b)$. So, by the
definition, $d_{BM}(tp(\bar b),\ell_p)\leq\epsilon$.  As
$\epsilon$ is arbitrary, there is a sequence $\bar a\in M$ (for
$\epsilon>0$) such that $d_{BM}(tp(\bar a),\ell_p)\leq\epsilon$.
This means that there is a sequence in $M$ which is
  $(1+\epsilon)$-isomorphic to $\ell_p$ (or $c_0$).
Again, as $\ell_2$ is finitely representable in $c_0$ and
$\ell_p$, $M$ embeds  $\ell_2$ almost isometrically.
\end{proof}

\begin{rmk} \label{stricter}
Recall that for two perturbation systems
$\mathfrak{p},\mathfrak{p}'$, the perturbation system
$\mathfrak{p}$ is stricter than $\mathfrak{p}'$, denoted by
$\mathfrak{p}<\mathfrak{p}'$, if for each $\epsilon>0$ there is
$\delta>0$ such that
$\mathfrak{p}(\delta)\subset\mathfrak{p}'(\epsilon)$. Clearly, if
$\mathfrak{p}<\mathfrak{p}'$ and $N, N$ are
$\mathfrak{p}$-isomorphic then they are
$\mathfrak{p}'$-isomorphic too. Therefore, if $\mathfrak{p}$ be
any perturbation system stricter than $\mathfrak{p}_{_{BM}}$, then
the ``almost isometrically" version of the above theorem holds for
$\mathfrak{p}$.
\end{rmk}

\begin{exa} \label{perturbed example} (i) Let $T=Th(\mathcal{AN}_{\subseteq[1,r]})$ be the theory of Nakano spaces presented
in \cite{Ben-Nakano}. The theory $T$ is not
$\aleph_0$-categorical, but it is $\lambda$-stable for
$\lambda\geqslant \mathfrak{c}$ (see Theorem~3.10.9 in
\cite{Poitevin}). Ben~Yaacov proved that $T$ is
$\aleph_0$-categorical and $\aleph_0$-stable up to small
perturbations of the exponent function. This perturbation system
is stricter than $\mathfrak{p}_{_{BM}}$ in the language of Banach
spaces (see Proposition~4.6 in \cite{Ben-Nakano}), and so the
theory of Nakano spaces is
$\mathfrak{p}_{_{BM}}$-$\aleph_0$-categorical.

\noindent
 (ii) (non-example) The theory of the 2-convexification
of Tsirelson space, is denoted by $M_{\mathcal{T},2}$,  as
presented by B. Johnson (in the language of Banach spaces) is not
$\mathfrak{p}_{_{BM}}$-$\aleph_0$-categorical because this space
does not contain  any $\ell_p$ or $c_0$. It is known  that every
ultrapower of this space is linearly homeomorphic to a canonical
direct sum of the space and a Hilbert space of suitable dimension.
Indeed, see Proposition~2.4.a  of \cite{JLS} for the proof; of
course  this depends on the argument of Pisier that is pointed to
in \cite{JLS} in the paragraph just before Proposition~2.4.a (from
the book \cite{Pisier}).
 Ward Henson suggested  the following argument that there are separable models $N$ of the theory of  $M_{\mathcal{T},2}$ such that $N$ and
$M_{\mathcal{T},2}\oplus\ell_2$ are isomorphic. Indeed, for an
ultrafilter ${\mathcal U}$, let $H$ be a subspace of
$(M_{\mathcal{T},2})_{\mathcal U}$ isomorphic to a Hilbert space
for which $(M_{\mathcal{T},2})_{\mathcal U}$ is the direct sum
internally of $M_{\mathcal{T},2}$ and $H$ (which is possible by
the result pointed above).
Let $H'$ be any separable, infinite dimensional subspace of $H$
and let $N$ be a separable infinite dimensional elementary
subspace of $(M_{\mathcal{T},2})_{\mathcal U}$ that contains
$M_{\mathcal{T},2}$ and $H'$. Letting $H_0 = Y \cap H$ it is easy
to show that $N = M_{\mathcal{T},2} \oplus H_0$ (again, the sum
$\oplus$ is the one internal to $(M_{\mathcal{T},2})_{\mathcal
U}$).  Of course this says essentially nothing about the precise
way in which the norm on $N= M_{\mathcal{T},2} \oplus H_0$ is
defined -- and it seems likely that there will be a large number
of different (non-isometric) possibilities.
 To summarize, $M_{\mathcal{T},2}$ and
$M_{\mathcal{T},2}\oplus  H_0$ are two non
$\mathfrak{p}_{_{BM}}$-isomorphic separable models of this theory.
\end{exa}

\begin{dfn} \label{2definition aleph_0}
A separable Banach space is called $\aleph_0$-categorical if it
is $\mathfrak{p}_{{\text{id}}}$-$\aleph_0$-categorical, where
$\mathfrak{p}_{{\text{id}}}$ is the strictest perturbation
system. (See also Definition~\ref{definition aleph_0} below.)
\end{dfn}

\begin{cor} \label{main theorem}  Every $\aleph_0$-categorical Banach space $M$ contains
isometric copies of $c_0$ or $\ell_p$ ($1\leqslant p<\infty$).
Moreover, $M$ contains an isometric copy of $\ell_2$.
\end{cor}

\begin{proof}
This is a consequence of \cite[Lemma~2.5]{Ben-perturb} and the
main theorem (Theorem~\ref{perturbed main theorem}) assuming the
identity perturbation system $\mathfrak{p}_{\text{id}}$.  Indeed,
by Theorem~\ref{perturbed main theorem}, for every $\epsilon>0$,
$p^{{\frak p}_{\text id}(\epsilon)}(\bar x^\epsilon)$ is realised
in $M$, where $p(\bar x)$ is the $0-\ell_p$-type (or
$0-c_0$-type). By Lemma~2.5 of \cite{Ben-perturb}, for every
$\epsilon>0$, $p^{{\frak p}_{\text id}(\epsilon)}(\bar x)$ is
realised in $M$. As ${\frak p}_{\text id}$ is the identity
system, i.e. $p^{{\frak p}_{\text id}(\epsilon)}=\{p\}$ for all
$\epsilon>0$, there is a sequence in $M$ which is isometric to
$\ell_p$ (or $c_0$). Similarly, $M$ contains an isometric copy of
$\ell_2$.
\end{proof}

\begin{rmk} (i) The spaces $L_p[0,1]$ ($1\leq p<\infty$) are
$\aleph_0$-categorical (see \cite{BBHU}). The Gurarij space (a
universal, ultrahomogeneous Banach space) is also
$\aleph_0$-categorical (see \cite{BH}).

\noindent
 (ii) Note that since the nature of our question, i.e.
the existence of good subspaces, is local and
$\aleph_0$-categoricity is a global condition on norm, one can
not expect a converse to Theorem~\ref{perturbed main theorem}.
For example, the space $M_{\mathcal{T}}\oplus\ell_p$ has a
complex norm (since it contains the Tsirelson space
$M_{\mathcal{T}}$), although it contains a good subspace, i.e.
$\ell_p$.

\noindent
 (iii) As we mentioned earlier, Corollary~\ref{main
theorem} is folklore to experts in the field. Ward Henson
informed us that one can prove that every $\aleph_0$-categorical
space contains $\ell_2$ using Dvoretzky's Theorem.
\end{rmk}

\begin{rmk} Note that sequence  spaces  $\ell_p$ ($p\neq 2$) and $c_0$
are not  $\aleph_0$-categorical. Despite this, $\ell_p$ ($p\neq
2$) is {\em almost} categorical, i.e. its theory has exactly two
separable models $\ell_p$ and $L_p[0,1]\oplus_p\ell_p$. This fact
is folklore to experts in field, including Alexander Berenstein,
Ward Henson and  Jose Iovino. We thank C. Ward Henson for
communicating to us the following argument and references.  The
main point behind this fact is the classification of $L_p(\mu)$
spaces under elementary equivalence in Theorem~2.2 of \cite{Hen}.
The proof has four steps. First, it was shown in
\cite[Corollary~2.5]{HenM1} that for a Banach space $M$ with
nonstandard hull $\widehat{M}$:

\medskip
$ \ \ \ \ \ \ \ \ \ \ \ \ \ \ \ \ \ \ M$ is an $L_p$-space if and
only if $\widehat{M}$ is an $L_p$-space.

\medskip\noindent
Therefore, if two Banach spaces $\ell_p$ and $M$ have isometric
nonstandard hulls, then $M$ is an $L_p$-space. (Recall that, by
Kakutani representation theorem,   every $L_p$-space $M$ is of
the form $L_p(\mu)$ for some measure $\mu$.)
 Second, by  the
Keisler-Shelah Theorem for metric structures
(\cite[Theorem~1.13]{Hen}), two metric structures are elementarily
equivalent iff they have linearly isometric ultrapowers iff they
have linearly isometric nonstandard hulls.  (We believe this
theorem should be credited to C. Ward Henson \cite{Hen} and to
Jacques Stern  \cite{Stern} equally and independently.) Using
this theorem, if $\ell_p$ and $M$ are elementary equivalent, then
$M$ is of the form $L_p(\mu)$ for some measure $\mu$. Third, by
Theorem~2.2 of \cite{Hen},  if $\ell_p$ and $L_p(\mu)$ are
elementary equivalent, then the measure space of $\mu$ has
infinitely many atoms. Fourth, if $M = L_p(\mu)$ is separable and
the measure space of $\mu$ has infinitely many atoms, then either
(1) the atomless part of the measure space is $\{0\}$ (so $M$ is
linearly isometric to $\ell_p$),  or (2) the atomless part of the
measure space is isomorphic to the measure space of $[0,1]$ with
Lebesgue measure (so $M$ is linearly isometric to $\ell_p \oplus_p
L_p[0,1]$). (The latter statement is a well-known fact due to
Carath\'{e}odory \cite[Section~14, Theorem~5]{Lacey}.) So putting
everything together we see that the complete theory of $\ell_p$
($p\neq 2$) has exactly two separable models $\ell_p$ and
$L_p[0,1]\oplus_p\ell_p$ up to isometry.  This fact suggests the
following conjecture: {\em Do all models of a theory with
finitely many separable models contain some $\ell_p$ or $c_0$?}
\end{rmk}

Still, one can say more:

\begin{rmk} \label{finitely representable} (i) Let $M$ be an  $\aleph_0$-categorical  Banach
space. Then $M$ embeds (isometrically) every Banach space with a
basis which is finitely representable in $M$. For this, repeat the
argument of the proof of Theorem~\ref{perturbed main theorem}.

\noindent
 (ii)  The {\em spectrum} $\Sigma(X)$ of a Banach space
$X$ consist of all $p\in[1,\infty]$ for which $X$ contains
$\ell_p^n$'s uniformly, that is, there exists $K>0$ such that for
all $n$, there is a subspace of $X$  which is $K$-isomorphic to
$\ell_p^n$. (See \cite[p. 57]{Schw}) 
By Krivine's theorem, the spectrum $\Sigma(X)$
 consist of all $p\in[1,\infty]$ for which
$\ell_p$ is finitely representable in $X$ (see \cite{Guer},
Theorem~II.5.13). Dvoretzky's theorem says that $2\in\Sigma(X)$.
It is known that $\Sigma(X)$ is a closed subset of $[1,\infty]$,
and $\Sigma(X)\cap[1,2]$ is always an interval; for example,
$\Sigma(L_p)=[p,2]$ if $1\leqslant p\leqslant 2$,
$\Sigma(L_p)=\{2,p\}$ if $2\leqslant p<\infty$, and
$\Sigma(L_\infty)=[1,\infty]$. Therefore:
 Every $\aleph_0$-categorical  Banach space $M$ contains
isometric copies of  $\ell_p$ for all $p\in \Sigma(M)$.
\end{rmk}

\subsection*{Concluding remarks}

(1) As previously mentioned, Krivine and Maurey \cite{KM} proved
that every {\em stable} Banach space contains a copy of $\ell_p$
for some $1\leqslant p<\infty$. Also, it was noticed that
 the type space of every separable stable Banach
space is {\em strongly separable}. Later, Haydon and Maurey
\cite{HM} showed that a Banach space with strongly separable
types contains either a reflexive subspace or a subspace
isomorphic to $\ell_1$.




\medskip\noindent
 (2)  In \cite{Cai},  Caicedo showed that for  any suitable extension $L$ of continuous logic with the compactness and the
separable downward L\"{o}wenheim-Skolem theoems, each sentence of
$L$ is equivalent to a sentence of continuous logic. On the other
hand, using these two theorems (in continuous logic), one can
give a proof of Corollary~\ref{main theorem}.  Indeed, suppose
that $M$ is $\aleph_0$-categorical. By Krivine's theorem, some
$\ell_p$  (or $c_0$) is  finitely representable in $M$, and so,
using the compactness theorem, there is an elementary extension
$N$ of $M$ such that it contains $\ell_p$ (or $c_0$). By the
separable downward L\"{o}wenheim-Skolem theorem, there is a {\em
separable}  elementary substructure $M'$ of $N$ such that it
contains $\ell_p$ (or $c_0$). By the $\aleph_0$-categoricity, $M$
and $M'$ are isometric, and so $M$ contains some classical
sequence space. As $\ell_2$ is finitely representable in $\ell_p$
and $c_0$, it is easy to see that $M$ contains $\ell_2$.
 Therefore, by Caicedo's result,
the above argument does not work in a logic {\em stronger} than
continuous logic.

 \medskip\noindent
 (3) There is still another stronger property: a separable Banach space $X$ is said to be {\em finitely determined}  if for
each separable space $Y$ such that $X$ is finitely representable
(f.r.) in $Y$ and $Y$ is f.r. in $X$ then $Y$ is isometric to
$X$. Clearly, every finitely determined space is
$\aleph_0$-categorical, but not converse. In \cite{K5}, a more
direct proof of existence of classical sequence spaces in finitely
determined spaces is given.

 \medskip\noindent
 (4) One might expect that $\aleph_0$-categoricity implies that
 a large number of $\ell_p$, $p\in[1,\infty)$, or $c_0$, are involved in the
 space. This is not true; for example, the theory $\ell_2\cong L_2=L_2([0,1],\lambda)$
is $\aleph_0$-categorical (see \cite{BBHU}), but $\ell_2$
 does not have any subspace isomorphic to $c_0$ or $\ell_q$ for all
 $q\neq 2$  (see \cite{AK}, Corollary~2.1.6).
 In other words, a direct proof of Theorem~\ref{perturbed main theorem}, without using Krivine's Theorem,
 does not imply a stronger result.

 \medskip\noindent
 (5) As previously mentioned,  Theorem~1.1  of \cite{Iov-2005} asserts that
the existence of enough definable types guarantees the existence
of $\ell_p$ subspaces.
Note that our approach in this paper is different. In fact we
counted separable models and the main result
(Theorem~\ref{perturbed main theorem}) is about complete theories.

 \medskip\noindent
 (6) Although $\aleph_0$-categoricity and
$\aleph_0$-stability do not have any connection in general, but to
our knowledge the most examples of $\aleph_0$-stable theories are
studied in continuous model theory are $\aleph_0$-categorical
(see \cite{BBHU}).

\medskip\noindent
 (7) (A Krivine-Maurey type theorem).
We believe that the following statement holds: For any separable
NIP space $X$ there exists a spreading model of $X$ containing
$c_0$ or $\ell_p$ for some $1\leqslant p<\infty$. The proof used
Borel definability of types and Krivine's Theorem. This result
provides answers to the questions in Remark~\ref{Gowers example}.
Full proof will be presented elsewhere.

\section{Dividing lines in Banach spaces and model theory} \label{correspondence}
This section aims to provide a classification of Banach spaces
similar to Shelah's classification in classical logic.  For the
sake of completeness, we recall some facts which were observed in
\cite{K4}, and then  we  use them to give some examples which are
(in our view) very illuminating.


\subsection*{Banach space for a formula} Let $\M$ be an
$L$-structure, $\phi(x,y):M\times M\to \Bbb R$ a formula (we
identify formulas with real-valued functions defined on models).

Let $S_\phi(M)$ be the space of complete $\phi$-types over $M$ and
set $A=\{\phi(x,a),-\phi(x,a)\in C(S_\phi(M)): a\in M\}$. The
(closed) convex hull of $A$, denoted by ($\overline{conv}(A)$)
$conv(A)$, is the intersection of all (closed) convex sets that
contain $A$. $\overline{conv}(A)$ is convex and closed, and
$\|f\|\leqslant \|\phi\|$ for all $f\in\overline{conv}(A)$. So,
by normalizing we can assume that $\|f\|\leqslant 1$ for all
$f\in\overline{conv}(A)$. Set $B=\overline{conv}(A)$ and
$V=\bigcup_{\lambda>0}\lambda B$. It is easy to verify that $V$
is a Banach space with the normalized norm and $B$ is its unit
ball. This space will be called the {\em space of linear
$\phi$-definable relations  on $M$}.

\begin{fct}[\cite{K4}] 
Assume that $\phi(x,y)$, $\M$, $B$ and $V$ are as above. Then the
following are equivalent:
\begin{itemize}
             \item [(i)]  $\phi$ is stable on $\M$.

             \item [(ii)]  $B$  is weakly compact.

             \item [(iii)]   The Banach space $V$ is reflexive.
\end{itemize}
\end{fct}

Recall that for an infinite compact Hausdorff space $X$, the
space $C(X)$ is not reflexive.

\begin{fct}[\cite{K4}]
Assume that $\phi(x,y)$, $\M$, and $V$ are as above. Then the
following are equivalent:
\begin{itemize}
             \item [(i)] $\phi$ has NIP on $\M$.

             \item [(ii)] $V$ is Rosenthal Banach space.

             \item [(iii)]  $V$ does not contain an isomorphic copy of $\ell_1$.
\end{itemize}
\end{fct}

\begin{dfn} \label{SOP} Let ${\M}$ be a  model (of theory $T$) and
$\phi(x,y)$ a formula. We say $\phi(x,y)$  has the {\em strict
order property on $\M$} (short SOP on $\M$) if there exists a
sequence $(a_ib_i:i<\omega)$ in  $\M$ and $\epsilon>0$ such that
for all $i<j$,
$$\phi({\M},a_i)\leqslant\phi({\M},a_{i+1}) \ \ \ \mbox{ and } \ \ \phi(b_j,a_i)+\epsilon<\phi(b_i,a_j).$$
The acronym SOP stands for the strict order property and NSOP is
its negation. A complete theory $T$ has  SOP if there are a
formula $\phi$ and a model $\M$ of it such that $\phi$ has SOP on
$\M$, and otherwise it is said that $T$ has NSOP.
\end{dfn}

\begin{rmk} (i)   Let ${\mathcal U}$ be a monster model (of theory
$T$), $\phi(x,y)$ a formula and $\M$ a small model. Suppose that
$\phi$ has SOP on  $\M$. Then  it is easy to verify that $\phi$
has SOP on the moster model $\mathcal U$. (Indeed, suppose that
$p$ is a $\phi$-type over $\M$, $(b_j)\in\M$ and $tp(b_j/{\M})\to
p$. By SOP on $\M$, $\phi(b_j,a_i)\leqslant\phi(b_j,a_{i+1})$,
and since $\phi(x,a_i)$'s are continuous, so
$\phi(b,a_i)\leqslant\phi(b,a_{i+1})$ for all $b\models p$. This
means that $\phi(b,a_i)\leqslant\phi(b,a_{i+1})$ for all
$b\in\mathcal U$.)

\noindent (ii) SOP strictly implies instability. In classical
($\{0,1\}$-valued) logic,  this is a known fact.

\noindent (iii) Suppose that $\M=\mathcal U$ and $V$ is as above.
If $V$ is weakly sequentially complete, i.e. every weak Cauchy
sequence has a weak limit (in $V$), then  $\phi$ has NSOP. This is
a consequence of the Eberlein--\v{S}mulian theorem (Fact~\ref{ES
theorem}). (Cf. \cite[Remark~3.4(ii)]{K-Baire}.)

\noindent (iv) Recently in \cite{K-Baire} it is shown that the
converse of (iii) above does not hold (for classical logic). That
is, if $\phi$ has NSOP, we cannot conclude that $V$ is weakly
sequentially complete (cf. \cite[Example~3.5(ii)]{K-Baire}). On
the other hand, we strongly believe that there is {\em no} 
perfect analog of Shelah's theorem (i.e. a complete theory is
stable if and only if it has NIP and NSOP) in {\em continuous}
logic.   This is discussed in detail in \cite{K-classification}.
 For the sake of
completeness we give the diagram of these observations for {\em
classical} logic \cite{K-Baire}:

\bigskip

\bigskip

\bigskip

~~~~~~~~~~~~~~~~~~~~~~~~~~~~ {\scriptsize Shelah}

~~ Stable ~~~~~~~~~~~~~~~~~ $\Longleftrightarrow$ ~~~~~~~~~~~~~
$NIP$ ~~~~~~~ $\&$ ~~~~~ $NSOP$

\bigskip

~~~~~ $\Updownarrow$ {\scriptsize Eberlein--Grothendieck}
~~~~~~~~~~~~~~~~~~~~~ $\Downarrow$
 {\scriptsize
 \ \ \ \ \ \ \ \ \ \ \ }
  ~~~~~~~~~~~~~~ $\Uparrow $

\medskip

~~~~~~~~~~~~~~~~~~~~~~~~ {\scriptsize Eberlein--\v{S}mulian}

 Reflexive ~~~~~~~~~~~~~~~~~ $\Longleftrightarrow$
~~~~~~~~~ Rosenthal ~~~~ $\&$ ~~~~~ $WS$-complete

\bigskip

\medskip
\end{rmk}


\subsection*{Dividing lines in classical spaces}  \label{examples}
In this subsection, we list some classical Banach spaces and some
theories and point out their model theoretic dividing lines. Many
of the following examples have been studied previously (mostly by
Henson and co-authors in the 70's), and much of what is said
about them is well-known (see \cite{HI} and references therein).
Nevertheless, there is a definite merit to gathering all of these
facts together in one place.

\medskip
It should be noted that since the notions NIP in a model (in both
classic and continuous logics) and SOP for continuous logic are
new, the following observations on these notions do not appear
somewhere in literature. In fact, in our view, the examples form
the most interesting and illuminating part of the article.

\begin{exa}  \label{exam1}
\begin{itemize}
             \item [(i)]  $L_p$ Banach lattices ($1\leqslant p<\infty$) are
             reflexive and so Rosenthal and weakly sequentially
             complete. By Kakutani representation theorem of $L$-spaces (\cite[4.27]{AB}),
             it was proved that the classes of $L_p$ Banach lattices are
             axiomatizable (in a suitable language) and their theories, denoted
             by $ALpL$, are stable (see \cite{BBHU}, Section~17).

             \item [(ii)] $L_1$ Banach lattices are neither reflexive nor Rosenthal in
             general (e.g. $\ell_1$). Although, they are weakly sequentially complete.
             Nevertheless, their theory, $AL1L$, is stable (see
             (i) above). However, it is not established in some extensions of the language. For
             example, Alexander Berenstein showed that $L_1({\Bbb R})$ with
             convolution is unstable (see \cite{Be}), later did it appear that
             $\ell_1(\mathbb{Z},+)$ with convolution is
             unstable (see \cite{FHS1}, Proposition~6.2). In fact it has SOP (see Example~\ref{exam2} below). On the other hand, since the theory atomless
             probability measure algebra, denoted by $APA$, is interpretable in the
             theory $L_1$ Banach lattices, the prior also is stable (\cite{BBHU}).

             \item [(iii)]  Hilbert spaces are reflexive. The theory infinite dimensional Hilbert spaces,
             as described in \cite{BBHU}, is stable.

             \item [(iv)] \label{c0}   $c_0$ is neither reflexive nor weakly
             sequentially complete, but it is Rosenthal. Let $(e_n)_1^\infty$ be
             standard basis for $c_0$ and $s_n=e_1+\cdots+e_n$.
             $(s_n)_1^\infty$ is called summing basis for $c_0$.
             Then $\|e_m+s_n\|=2$ if $m\leqslant n$ and $=1$ if
             $m>n$. So, the formula $\phi(x,y)=\|x+y\|$ has order
             property. Now let $x=(a_n)_1^\infty\in\ell_1=(c_0)^*$
             then $x^*(s_n)\to \sum_1^\infty a_k$. So $(s_n)_1^\infty$ is a weak Cauchy sequence with no weak
             limit. It is trivial that $c_0$ is Rosenthal, because $(c_0)^*=\ell_1$ is separable but $(\ell_1)^*=\ell_\infty$ is nonseparable. Let $\psi(x,s_n)=\max(\|x+s_n\|,\|x-s_n\|)$.
             Then $\psi(x,s_m)\leqslant\psi(x,s_n)$ and
             $\psi(e_n,s_m)+1\leqslant\psi(e_m,s_n)$ for all $m<n$. So
             $\psi(x,y)$ has SOP on $c_0$.

             \medskip
             By Example~\ref{c_0,1} above, $\phi(x,y)$ has IP on $c_0$. Therefore, if $c_0$ is a model of a theory $T$, then $T$ has
             IP. (Also, see Example~\ref{c_0} for more information.)

             \item [(v)] $\ell_\infty$ is neither stable nor
             Rosenthal or weakly sequentially complete because $c_0$ and $\ell_1$ live
             inside $\ell_\infty=C(\beta\Bbb N)$ (see below). The formula $\psi$ as described above has SOP
             in $\ell_\infty$.  This fact explains why the case $p=\infty$ is excluded from (i) above.
             Since $c_0$ has IP, $\ell_\infty$ has IP.

             \item [(vi)] $C(X,{\Bbb R})$ space for an (infinite) compact
             Hausdorff space is neither reflexive nor
             Rosenthal nor weakly sequentially complete (see (vii) below). By
             Kakutani representation theorem of $M$-spaces, the
             class of $C(X)$-spaces is axiomatizable in the language of Banach lattices. Indeed,
             replace the axioms of abstract $L_P$-spaces in \cite{BBHU} by the $M$-property
             $\|x^++y^+\|=\max(\|x^+\|,\|y^+\|)$.
             Now we add a constant symbol $\textbf{1}$ to our language
             as an order unite, i.e. for all $x$, $\|x\|=1$ implies that $x^+\leqslant
             \textbf{1}$; equivalently,
             $\sup_x((\|x\|-1)\dotminus\|x^+\vee{\bf 1}-{\bf
             1}\|)$. Note that $\|\cdot\|$ is an order unite norm
             if for all $x$ and $\alpha$, $x^+\leqslant \alpha{\bf
             1}$ implies that $\|x\|\leqslant\alpha$;
             equivalently $\sup_x(\|x^+\vee\alpha{\bf 1}-{\bf
             1}\|\dotminus(\|x\|-\alpha))$. Since $\ell_\infty$
             has SOP, the theory of $C(X)$-spaces has SOP.
             Since $\ell_{\infty}$ has IP, then the theory of $C(X)$-spaces has the independence property.

             \item [(vii)]  $C_0(X,{\Bbb C})$ for an (infinite) locally compact space $X$  is neither reflexive nor
             Rosenthal nor weakly sequentially complete. Note that, by Gelfand-Naimark
             representation theorem of Abelian $C^*$-algebras, the class
             $C_0(X,{\Bbb C})$-spaces is axiomatizable in a suitable language (see \cite{FHS1} for the noncommutative theory).
             First, it is known that  all Banach spaces, in particular $\ell_1$ and
             $c_0$, live inside $C(K)$-spaces where $K$'s are compact. So
             $C(K)$ is neither reflexive nor
             Rosenthal or weakly sequentially complete in general.
             Second, by an example of \cite{FHS1}, we will show that the theory $C_0(X)$-spaces has SOP and IP. Indeed, in $X$ find a sequence of distinct $x_n$ that
             converges to some $x$ (with $x$ possibly in the
             compactification of $X$). For each $n$ find a positive $a_n\in
             C_0(X)$ such that $\|a_n\|=1$ and $a_n(x_i)=1$ if
             $i\leqslant n$ and $=0$ if $i>n$. By replacing $a_n$
             with maximum of $a_j$ for $j\leqslant n$, we may
             assume that this sequence is increasing. Also, we can
             assume that the support $K_n$ of each $a_n$ is
             compact and $a_m(K_n)=1$ for $n<m$. So $a_ma_n=a_n$
             if $n\leqslant m$. Now let $\phi(x,y)=\|(x-1)y\|$.
             Then $\phi(x,a_i)\leqslant\phi(x,a_j)$ and
             $\phi(a_j,a_i)+1\leqslant\phi(a_i,a_j)$ for $i<j$. So
             $\phi$ has SOP. Since $c_0$ has IP, this theory has IP.

             \item [(viii)] $C^*$-algebras are neither reflexive nor
             Rosenthal nor weakly sequentially complete because
             some infinite-dimensional Abelian $^*$-subalgebra lives inside a $C^*$-algebra and so by
             Gelfand-Naimark theorem it has a $C_0(X)$-subalgebra (see \cite{FHS1},
             Lemma~5.3). Note that the Gelfand-Naimark-Segal
             theorem ensure that this class is axiomatizable
             (\cite{FHS1}). So, the theory of $C^*$-algebras has
             SOP. Since $c_0$ has IP, this theory has IP.

             \item [(ix)] The class of (Abelian) tracial von
             Neummman algebra is axiomatizable (see \cite{FHS2}).
             A tracial von Neumman algebra is stable if and only
             if it is of type I. The theory of Abelian tracial von Neumann algebras is stable
             because it is interpretable in the theory of
             probability measure algebras (see above and also
             \cite{FHS1}, Lemma~4.5).
\end{itemize}
\end{exa}

To our knowledge the following observation itself does not appear
somewhere in literature.

\begin{exa}[$\ell_1({\Bbb Z},+)$ with convolution has SOP]  \label{exam2}
In \cite{FHS2} it is shown that the formula
$\phi(x,y)=\inf_{\|z\|\leqslant 1}\|x*z-y\|$ has order property in
$\ell_1({\Bbb Z},+)$. We show that $\phi$ also has SOP. Indeed, we
need to show that: (i) $\phi(x_i,y)\leqslant \phi(x_j,y)$ for
$i\leqslant j$, and (ii) $\phi(x_i,x_i)<\phi(x_j,x_i)$ for $i<j$,
where $x_i$'s are elements of $\ell_1$ as described in
\cite[Proposition~6.2]{FHS2}. It is easy to verify that (ii)
holds. We check that (i) also is true. Indeed, for any $z$,
$\|z\|\leqslant 1$ and any $y$, we have
$|\|y\|-\|z\||\leqslant\sup_t|z(t)-y(t)|\leqslant\|y\|+1$  where
$\|\cdot\|$ is uniform norm on $C({\Bbb T})$.  So, if we assume
that $\|y\|\geqslant 1$, then  $\inf_{\|z\|\leqslant
1}\|z-y\|=|\|y\|-1|$. (Note that we can replace $y$ with
$\max(y,1)$ and so assume that $\|y\|\geqslant 1$. In fact we
define $\phi(x,y)=\inf_{\|z\|\leqslant 1}\|x*z-\max(y,1)\|$.)
Therefore if $i\leqslant j$, then $\inf_{\|z\|\leqslant 1}\|(\cos
t )^{2i}z-y\|\leqslant\inf_{\|z\|\leqslant 1}\|(\cos
t)^{2j}z-y\|$. Equivalently, $\phi(x_i,y)\leqslant \phi(x_j,y)$.
To summarize, $\phi$ has SOP.  A~question arises: {\em Does
$\ell_1({\Bbb Z},+)$ with convolution have NIP?}
\end{exa}

The following interesting fact affirms some of our observations.

\begin{fct}[\cite{Guer}, Theorem III.5.1] Every stable (separable) Banach space is weakly sequentially
complete.
\end{fct}

So, every Banach space containing $c_0$ is unstable. This
directly implies that the theories $C(X)$-spaces, $C_0(X,{\Bbb
C})$-spaces, and $C^*$-algebras are not stable (see (vi)--(viii)
above). (There is even something stronger (see
Proposition~\ref{c_0->IP}).) Of course, the converse does not
hold, e.g. Tsirelson's space  and its dual are unstable but they
are reflexive and so weakly sequentially complete.

\begin{rmk} In \cite{M} the author proved that for $1\leqslant
p<\infty$, $p\neq 2$, the non-comutative $L_p$ space
$L_p({\mathcal M})$ is stable iff $\mathcal M$ is of type I. (See
also \cite[p. 1479]{Hand2}.)  Later, in \cite{FHS1} the authors
showed that a tracial von Neumman algebra is stable if and only
if it is of type I. Two questions arise: {\em Which tracial von
Neumman algebra (or non-comutative $L_p$ spaces) have NIP? Which
of them have NSOP?}
\end{rmk}


\section{Appendix~A: $\aleph_0$-categoricity, continued}   \label{aleph_0
categorical}
In this appendix, we revisit the notion
$\aleph_0$-categoricity. Some observations are not new for model
theorists, but one reason for recalling them is to make the paper
more accessible to other interested readers.


\medskip
Recall that the density character of a topological space $X$, is
the least infinite cardinal number of a dense subset of $X$. When
measuring the size of a structure we will use its density
character (as a metric space), denoted $\|M\|$, rather than its
cardinality. Similarly, for a separable structure $M$, since
$S_{\phi}(M)$ is a metric space, we measure the size
$S_{\phi}(M)$ by its density character $\|S_{\phi}(M)\|$.

\begin{dfn} \label{definition aleph_0}  A complete theory $T$ in a countable language is {\em
$\aleph_0$-categorical} (or {\em $\omega$-categorical}) if it has
an infinite model and any two  models  of size $\aleph_0$ are
isomorphic. Equivalently, $T$ has an infinite model and {\em
separable}  models are isometric. An {\em $\aleph_0$-categorical}
(or {\em $\omega$-categorical}) {\em structure} is a separable
structure $M$ whose theory is $\aleph_0$-categorical. (Compare
with Defintition~\ref{2definition aleph_0}.)
\end{dfn}

Let $L$ be a language and $T$ a complete $L$-theory. Suppose that
$M$ is a model of  $T$ and $A\subseteq M$. Denote the
$L(A)$-structure $(M,a)_{a\in A}$ by $M_A$, and set $T_A$ to be
the $L(A)$-theory of $M_A$.

\begin{rmk} It is easy to check that for a separable model $M$, and
countable subsets $A\subseteq B$ of $M$, if $T_B$ is
$\aleph_0$-categorical then $T_A$ is also $\aleph_0$-categorical
(see \cite{BBHU}, Corollary~12.13); however, the converse does
not hold in a strong form (see Remark~12.14 in \cite{BBHU}).
\end{rmk}







The following fact is folklore, and we remove the proof. Recall
that a formula $\phi$ is stable for a theory $T$ if for every
model $M$ of $T$, $\phi$ is stable on $M$; equivalently, $\phi$
has not the order property   on every model $M$ of $T$.

\begin{fct}[\cite{Ben-Gro}, Corollary~4] \label{stable<->strong}  Let $T$ be a countable theory and $\phi(x,y)$ a formula.
Then the following are equivalent.
\begin{itemize}
             \item [(i)] For every separable model $M$ of $T$, $S_\phi(M)$ is strongly separable.
             \item [(ii)] $\phi$ is stable on every separable model of $T$.
             \item [(iii)] $\phi$ is stable for $T$.
             \item [(iv)] For every model $M$ of $T$, $\|S_\phi(M)\|\leqslant\|M\|$.
\end{itemize}
\end{fct}

\begin{cor} \label{stable=separable}
Let $M$ be an $\aleph_0$-categorical structure. The complete
theory $T$ of $M$ is stable if and only if for every formula
$\phi$, the space $S_\phi(M)$ is strongly separable.
\end{cor}

\begin{proof} If $T$ is stable, every $\phi$-type is definable (\cite[Corollary~4]{Ben-Gro}),
and so $S_\phi(M)$ is strongly separable. The converse follows
from the $\aleph_0$-categoricity and the equivalence
(i)~$\Leftrightarrow$~(iii) of Fact~\ref{stable<->strong} above.
\end{proof}

The fact that Tsirelson's space can not be $\aleph_0$-categorical
is a consequence of Corollary~\ref{main theorem}. Alternatively,
it is clear by instability and the above fact:

\begin{cor}   \label{strong-not-categorical}
Suppose that $T$ is the complete theory of Tsirelson's space
$M_{\mathcal{T}}$ (in a countable language). Then the followings
hold:

(i) $T$ is not $\aleph_0$-categorical.

(ii) There exists a separable model $M$ of $T$ which its type
space is not strongly separable.
\end{cor}

\begin{proof} (i): Suppose, if possible, that $T$ is
$\aleph_0$-categorical. Then, since $S_\phi(M)$ is
$\rho$-separable where $\phi(x,y)=\|x+y\|$, so $\phi(x,y)$ is
stable on $M_{\mathcal{T}}$. By  Corollary~\ref{unstable}, this
is a contradiction.

 (ii): Immediate from
Fact~\ref{stable<->strong} and Corollary~\ref{unstable}.
(Clearly, this model is different from Tsirelson's space by
Fact~\ref{Odell-fact}.)
\end{proof}

\begin{rmk} Note that the $\aleph_0$-categorical assumption in
Corollary~\ref{stable=separable} is too strong and a weaker
assumption is sufficient. Indeed, suppose that $M$ is a  ${\frak
p}_{BM}$-saturated structure, where ${\frak p}_{BM}$ is the
Banach--Mazur perturbation system. (See
Definitions~\ref{Banach--Mazur perurbation}, \ref{p-saturation}
above.) Then, it is easy to show that, $M$ is  stable if and only
if for every formula $\phi$ the space $S_\phi(M)$ is strongly
separable.
\end{rmk}

\section{Appendix~B: A remark on Rosenthal's dichotomy}   \label{appendix}

Note that a well-known result of Rosenthal (Lemma~\ref{lemma-1}
below) is used in the proof of the direction
(i)~$\Rightarrow$~(v) of Lemma~\ref{equivalence} (see Lemma~3.12
in \cite{K3}). So it is worthwhile to study it more carefully.

\begin{dfn} \label{indep} \begin{itemize}
             \item [(i)] A sequence $\{f_n\}$ of real valued functions on a set $X$ is said to be {\em independent }  if there exist real numbers $s<r$ such that
                         $$\bigcap_{n\in P} f_n^{-1}(-\infty,s)\cap\bigcap_{n\in M} f_n^{-1}(r,\infty)\neq\emptyset \ \ \ \ \ \ \ \  \boxtimes$$
                         for all finite disjoint subsets $P,M$ of $\mathbb
                         N$. A family $F$ of real valued functions on $X$ is called {\em independent} if it contains an independent sequence; otherwise it is called
                         {\em strongly} (or {\em completely}) {\em dependent}.
             \item [(ii)] We say that the sequence $\{f_n\}$ is {\em strongly} (or {\em completely})
                          {\em independent} if $\boxtimes$ holds for all {\em infinite}
                          disjoint subsets $P,M$ of $\Bbb N$. A family $F$ of real valued functions is called {\em strongly} (or {\em completely})
                          {\em independent} if it contains a strongly independent sequence; otherwise it is called {\em dependent}.
\end{itemize}
\end{dfn}

It is an easy exercise to check that when $X$ is a compact space
and functions are continuous the above two notions are the same,
but this does not hold in general (see Example~\ref{c_0} below).

\begin{lem} \label{dependent=strong dependent}
Let $X$ be a {\em compact} space and $F\subseteq C(X)$ a bounded
subset. Then the following conditions are equivalent:
\begin{itemize}
             \item [(i)]  $F$ is dependent.
             \item [(ii)] $F$ is strongly dependent.
\end{itemize}
\end{lem}

\begin{lem}[Rosenthal's lemma] \label{lemma-1} Let $X$ be a compact space and $F\subseteq C(X)$ a bounded subset. Then the following conditions are equivalent:
\begin{itemize}
             \item [(i)]  $F$ does not contain an independent subsequence.
             \item [(ii)] Each sequence in $F$ has a  convergent subsequence
in $\mathbb R^X$.
\end{itemize}
\end{lem}

Rosenthal \cite{Ros} used the above lemma for proving his famous
$\ell_1$ theorem: a sequence in a Banach space is either `good'
(it has a subsequence which is weakly Cauchy) or `bad' (it
contains an isomorphic copy of $\ell_1$). We will shortly discuss
this topic (see below).

\begin{fct}[\cite{Ros}] \label{fact3} If $X$ is a compact space and
$\{f_n\}$ a pointwise bounded sequence in $C(X)$, then
\begin{itemize}
             \item [(1)]  either $\{f_n\}$ has a (pointwise) convergent subsequence, or

             \item [(2)] $\{f_n\}$ has a $\ell_1$-subsequence, equivalently, $\{f_n\}$
has an independent subsequence.
\end{itemize}
\end{fct}

Note that for non-compact spaces, Fact~\ref{fact3} is not a
dichotomy. Michael Megrelishvili
 informed us with details about this
observation (see  Example \ref{c_0,1} above and Example \ref{c_0}
below)  and also pointed out to an Example which is in
\cite{Dulst}. Of course, we are not sure that the observation
itself (that Fact~\ref{fact3} is not always dichotomy in
noncompact Polish case) appears somewhere in the literature in a
clear form. For compact space, independence and strong
independence are the same.

\begin{fct}[Rosenthal's $\ell_1$-theorem] \label{fact1}  If $(x_n)$ is a bounded sequence in a Banach
space $V$ then
\begin{itemize}
             \item [(1)$'$]  either $(x_n)$ has a weakly Cauchy subsequence, or
             \item [(2)$'$] $(x_n)$ has a $\ell_1$-subsequence.
\end{itemize}
\end{fct}

 Fact~\ref{fact1} is a consequence of
Fact~\ref{fact3}. Indeed, let $B^*$ be the unit ball of $V^*$,
and define $f_n:B^*\to{\Bbb R}$ by $x^*\mapsto x^*(x_n)$ for all
$x^*\in B^*$.

\begin{fct}[Rosenthal's dichotomy, \cite{Tod}] \label{fact2} If $X$ is a Polish space and
$\{f_n\}$ a pointwise bounded sequence in $C(X)$, then
\begin{itemize}
             \item [(1)$''$]  either $\{f_n\}$ has a (pointwise) convergent subsequence, or
             \item [(2)$''$] $\{f_n\}$ has a subsequence whose its closure in ${\Bbb
R}^X$ is homeomorphic to $\beta\Bbb N$, equivalently, it has a
strong independent subsequence.
\end{itemize}
\end{fct}

\begin{rmk} Note that (1)$'$~$\Rightarrow$~(1)=(1)$''$, but
(1)$''$~$\nRightarrow$~(1)$'$ in general. Also,
(2)$''$~$\Rightarrow$~(2)$'$, but (2)$'$~$\nRightarrow$~(2)$''$
(see below). In fact (2)$'$ is equivalent to independence
property.  For Polish space $X$, (2)$''$ holds iff there is a
compact subset $K\subseteq X$ such that $\{f_n|_K\}$ contains a
strong independent subsequence, equivalently, $\{f_n\}$ has a
strong independent subsequence (see \cite{BFT}, Lemma~2B,
Theorem~2F and Corollary~4G).
\end{rmk}

\begin{exa}  \label{c_0} We revisit Example
\ref{c_0,1}. Recall that the family ${\mathcal F}_{c_0}$ (in
Example~\ref{c_0,1}) contains an independent sequence.  (i) The
family ${\mathcal F}_{c_0}$ is weakly precompact in ${\Bbb
R}^{B_{c_0}}$. Indeed, since $c_0$ is separable and the functions
$f_a$ are $1$-Lipschitz, so by a diagonal argument (see
Lemma~\ref{diagonal-lemma}), every sequence has a point-wise
convergent subsequence. This shows that Fact~\ref{fact3} is not a
dichotomy for non-compact spaces. (ii) By Fact~\ref{fact1}, the
sequence $(f_n)$ has not a weakly Cauchy subsequence. (iii) By
Fact~\ref{fact2}, $(f_n)$ has not a strong independent
subsequence.
\end{exa}

\begin{rmk} (i) Define $f_n:B_{c_0}\to[0,2]$ by
$f_n(x)=\|x+s_n\|$. Then $f_n$ converges to the continuous
function $f(x)=1+\|x\|$. Now, let $\widehat{c_0}$ be the
\v{C}ech-Stone compactification of $c_0$ and $\hat{f_n},\hat{f}$
be the extensions of $f_n,f$ respectively. Then
$\hat{f_n}\nrightarrow \hat{f}$. Indeed, note that
$\hat{f_n}|_{c_0}=f_n$ and $\hat{f}|_{c_0}=f$, and
$\hat{f}(e)=\lim_m\hat{f}(e_m)=\lim_mf(e_m)=\lim_m\lim_nf_n(e_m)=2\neq
1=\lim_n\lim_mf_n(e_m)=\lim_n\lim_m\hat{f_n}(e_m)=\lim_n\hat{f_n}(e)$
where $(e_m)$ is the standard basis of $c_0$ and $e$ is a cluster
point of $(e_m)$ in $\widehat{c_0}$. Note that $c_0$ is dense in
$\widehat{c_0}$ and $f_n\to f$, but
${\hat{f_n}}\nrightarrow\hat{f}$. (ii) Define $g_n(x)=\|x+e_n\|$.
Then $g_n\to g$ where $g(x)=\max(\|x\|,1)$. But $(\hat{g_n})$
does not contain a convergent subsequence; because $(g_n)$ has
independence property (see Example~\ref{c_0} above). (iii) $(g_n)$
is a poinwise convergent sequence but it has not a weakly Cauchy
subsequence. Because, by Rosenthal $\ell_1$-theorem, either a
sequence has a weakly Cauchy subsequence, or it has a
$\ell_1$-subsequence (equivalently, it contains an independent
subsequence). Note that weakly Cauchy is stronger than pointwise
convergence, in general. But, for locally compact spaces, they
are equivalent.
\end{rmk}

\end{document}